\documentclass[leqno, a4paper]{amsart}

\usepackage{babelbib}
\usepackage[USenglish]{babel}
\usepackage{amssymb}
\usepackage[all]{xy}

\theoremstyle{plain}
\newtheorem{lem}{Lemma}[section]
\newtheorem{prop}[lem]{Proposition}
\newtheorem{cor}[lem]{Corollary}
\newtheorem{thm}[lem]{Theorem}
\newtheorem*{thm*}{Theorem}
\newtheorem*{prop*}{Proposition}
\newtheorem*{cor*}{Corollary}
\newtheorem*{lem*}{Lemma}

\theoremstyle{definition} 
\newtheorem{ex}[lem]{Example}

\newtheorem{rem}[lem]{Remark}

\newtheorem{con}{Construction}[section]
\newtheorem{defn}{Definition}[section]
\newtheorem*{fact}{Fact}

\DeclareMathOperator{\Aut}{\mathrm{Aut}}
\DeclareMathOperator{\Hom}{\mathrm{Hom}}
\DeclareMathOperator{\Char}{\mathrm{char}}
\DeclareMathOperator{\Spec}{\mathrm{Spec}}
\DeclareMathOperator{\Gal}{\mathrm{Gal}}
\DeclareMathOperator{\Sm}{\mathrm{Sm}}
\DeclareMathOperator{\Deg}{\mathrm{deg}}
\DeclareMathOperator{\Res}{\mathrm{Res}}
\DeclareMathOperator{\Var}{\mathrm{Var}}

\DeclareMathOperator{\Dim}{\mathrm{dim}}
\DeclareMathOperator{\Ch}{\mathrm{ch}}
\DeclareMathOperator{\Td}{\mathrm{td}}
\DeclareMathOperator{\Bla}{\mathrm{mod}}
\DeclareMathOperator{\Sch}{\mathrm{Sch}}
\DeclareMathOperator{\Sets}{\mathrm{Sets}}
\DeclareMathOperator{\Ker}{\mathrm{ker}}
\DeclareMathOperator{\Id}{\mathrm{id}}

\DeclareMathOperator{\Pol}{\text{\![\![\!}}
\DeclareMathOperator{\Por}{\text{\!]\!]}}
\DeclareMathOperator{\Mel}{\text{\!(\!(\!}}
\DeclareMathOperator{\Mer}{\text{\!)\!)}}
\DeclareMathOperator{\Em}{{\mathfrak{m}}}

\begin{document}
\parindent 0em

\title{On rational points of varieties over local fields having a model with tame quotient singularities}
\author{Annabelle Hartmann}
\date{}

\begin{abstract}
We study rational points on a smooth variety $X$ over a complete local field $K$ with algebraically closed residue field,
and models $\mathcal{X}$ of $X$ with tame quotient singularities.
If $\mathcal{X}$ is the quotient of a Galois action on a weak N\'eron
model of the base change of $X$ to a tame Galois extension of $K$,
then we construct a canonical weak N\'eron model of $X$ with a map to
$\mathcal{X}$, and examine its special fiber.
As an application we get examples of singular models $\mathcal{X}$ such that there are $K$-rational points of $X$ specializing to a singular point of $\mathcal{X}$.
Moreover we obtain formulas for the motivic Serre invariant and the rational volume,
and the existence of $K$-rational points on certain $K$-varieties with potential good reduction.
\end{abstract}

\maketitle

\section{Introduction}

In this article we study smooth and proper varieties over a complete local field $K$ with algebraically closed residue field $k$
with regard to the existence
of $K$-rational points.

A standard way to detect rational points of varieties over complete local fields is to look at models.
A model of a $K$-variety $X$ is an integral, flat scheme $\mathcal{X}$ over the ring of integers $\mathcal{O}_K$ such that
the generic fiber of $\mathcal{X}$ is isomorphic to $X$.
There is a natural map $\mathcal{X}(\mathcal{O}_K)\to X(K)$, 
and a specialization map $\mathcal{X}(\mathcal{O}_K)\to \mathcal{X}_k(k)$, where $\mathcal{X}_k\subset \mathcal{X}$ is the special fiber.
If $\mathcal{X}$ is a proper $\mathcal{O}_K$-scheme, then the natural map $\mathcal{X}(\mathcal{O}_K)\to X(K)$ is a bijection.
As $\mathcal{O}_K$ is Henselian, the specialization map is surjective whenever $\mathcal{X}$ is smooth over $S:=\Spec(\mathcal{O}_K)$.

If $\mathcal{X}$ is regular, then every $\mathcal{O}_K$-point of $\mathcal{X}$ factors through the smooth locus of $\mathcal{X}$ over $S$, see \cite[Chapter 3.1, Proposition 2]{MR1045822}.
Hence if a $K$-variety $X$ has a regular and proper model $\mathcal{X}\to S$,
then $X$ has a $K$-rational point if and only if the special fiber of the smooth locus of $\mathcal{X}$ over $S$ is not empty.
But if $\mathcal{X}$ is not regular, then there may exist $\mathcal{O}_K$-points intersecting the singular locus of $\mathcal{X}$ over $S$,
see Example \ref{exsec}.

The existence of weak N\'eron models plays an important role in the study of rational points.
A weak N\'eron model of a smooth $K$-variety $X$ is a smooth and separated model $\mathcal{Z}$ of $X$, such that the natural map from $\mathcal{Z}(\mathcal{O}_K)$ to $X(K)$ is a bijection.
Hence if $X$ admits a weak N\'eron model, then $X$ has a $K$-rational point if and only if the special fiber of this weak N\'eron model is not empty.
It is known that every smooth and proper $K$-variety has a weak N\'eron model, see \cite[Chapter 3.5, Theorem~2]{MR1045822}.
But in general a weak N\'eron model is not unique.

The smooth locus over $S$ of a regular, proper model of a smooth, proper $K$-variety $X$ is a weak N\'eron model of $X$.
There is a way to obtain a weak N\'eron model from any proper model, the so called N\'eron smoothening, see \cite[Chapter~3]{MR1045822},
which is constructed by blowing up singular points having sections through them. 
But given a singular point,
it is hard to decide a priori whether there is a section containing that point. 
Therefore the N\'eron smoothening does not yield a straightforward method for constructing a weak N\'eron model from an arbitrary singular model.

\medskip

In this article we consider the following situation.
Let $X$ be a $K$-variety,
let $L/K$ be a Galois extension,
and let $X_L$ be the base change of $X$ to $L$.
Then $G\!:=\Gal(L/K)$ acts on $X_L$ such that $X_L/G\cong X$.
Consider a model $\mathcal{Y}$ of $X_L$ with a good $G$-action, i.e.~an action such that every orbit is contained in an affine open subscheme of $\mathcal{Y}$, extending this action on $X_L$.
Then the quotient $\mathcal{X}\!:=\mathcal{Y}/G$ is an $\mathcal{O}_K$-scheme and in fact a model of $X$.
In general $\mathcal{X}$ will have tame quotient singularities,
and there can be $\mathcal{O}_K$-points through the singular locus, see Example \ref{the example} and Example \ref{exsec}.

Note that interesting models of $X_L$ with an action as required really exist and appear naturally.
For example models of $X_L$ obtained from models of $X$ by base change and normalization have such a $G$-action, and
these are exactly the techniques used to construct a model with semistable reduction in the semistable reduction theorem, 
see \cite[Chapter~10, Proposition~4.6]{ MR1917232}.
Moreover, we show in Theorem~\ref{gaction on wnm}
that if $X$ is a proper and smooth $K$-variety, then there is always a weak N\'eron model of $X_L$
to which the Galois action on $X_L$ extends.
To construct such a weak N\'eron model,
we show in particular that the N\'eron smoothing as constructed in \cite[Chapter 3]{MR1045822} is compatible with actions of the Galois group as described above.

As the model $\mathcal{X}$ obtained by taking the quotient is singular in general and has sections through the singular locus,
neither $\mathcal{X}$ nor its smooth locus over $S$ will be a weak N\' eron model of $X$.
But there is a way to construct a weak N\'eron model of $X$ out of a weak N\'eron model of $X_L$ with a $G$-action extending the Galois action on the generic fiber.
In this context we show the following theorem.

\begin{thm*}
\emph{(Theorem \ref{main theorem})}
Let $L/K$ be a tame Galois extension, 
let $X$ be a smooth $K$-variety, and let $X_L$ be the base change of $X$ to $L$.
Let $ \mathcal{Y}$ be a smooth model of $X_L$ with a $G:=\Gal(L/K)$-action extending the Galois action on $X_L$.
Let $\mathcal{X}:= \mathcal{Y}/G$ be the quotient.\\
Then there is a smooth model $\mathcal{Z}$ of $X$ and a separated $S:=\Spec(\mathcal{O}_K)$-morphism ${\Phi: \mathcal{Z}\to \mathcal{X}}$,
such that the induced map $\mathcal{Z}(\mathcal{O}_K)\to \mathcal{X}(\mathcal{O}_K)$ is a bijection, and such that
for all smooth, integral $S$-schemes $\mathcal{V}$ and
all dominant $S$-morphisms $\Psi: \mathcal{V}\to \mathcal{X}$ there is a unique S-morphism $\Psi':\mathcal{V}\to \mathcal{Z}$
such that $\Phi \circ \Psi'=\Psi$.
In particular $\mathcal{Z}$ is unique.\\
If $\mathcal{Y}$
is a weak N\'eron model of $X_L$, then $\mathcal{Z}$ is a weak N\'eron model of $X$.
\end{thm*}

In fact, $\mathcal{Z}$ is the fixed locus of some $G$-action on the Weil restriction of $\mathcal{Y}$ to $S$, see Construction~\ref{construction Z}.
The construction goes back to \cite{ MR1149171}, where it is used in the context of abelian varieties and N\'eron models.

Note that the uniqueness of $\mathcal{Z}$ with its properties is interesting, because in general a weak N\'eron model is, in contrast to a N\'eron model, not unique.

The theorem and its proof also yield an explicit description of a weak N\'eron model $\mathcal{Z}$ of $X$.
Having this description at hand, we can examine its special fiber $\mathcal{Z}_k$ which is important for finding $K$-rational points of $X$.
We show the following key lemma.

\begin{lem*}\emph{(Lemma \ref{main lemma})}
Let $\mathcal{Y}^G$ be the fixed locus of the $G$-action on $\mathcal{Y}$.
Then there is a $k$-morphism
$b: \mathcal{Z}_k\to \mathcal{Y}^G$
such that for any point $y\in \mathcal{Y}^G$ with residue field $\kappa(y)$ 
the inverse image of $y$ is isomorphic to $\mathbb{A}^{m}_{\kappa(y)}$ as $\kappa(y)$-schemes for some $m\in \mathbb{N}$.
\end{lem*}

To show this lemma we use the explicit description of $\mathcal{Z}$ and of the $G$-action on the complete local ring of a fixed point,
which is examined in Lemma~\ref{ap1} and Lemma~\ref{ap2}.

There are some interesting applications of the key lemma.
For example we deduce from it\ that the quotient $\mathcal{X}$ has $\mathcal{O}_K$-points if and only if $\mathcal{Y}^ G\neq \emptyset$, see Corollary \ref{sections of the quotient}.
In fact these $\mathcal{O}_K$-point will pass through the image of $\mathcal{Y}^ G$ in $\mathcal{X}$, which in general will be singular.
Hence we obtain examples of singular models with section through the singular locus.

We can use the obtained results also to study certain motivic invariants, the motivic Serre invariant and the rational volume.
The motivic Serre invariant $S(X)$ of a $K$-variety $X$
is defined to be the class of the special fiber of a weak N\'eron model of $X$ in some quotient of the Grothendieck ring of varieties, namely in $K_ 0^{\mathcal{O}_K}(\Var_k)/(\mathbb{L}-1)$, see Definition \ref{def serre invariant}.
The Serre invariant is interesting in the context of rational points, because it vanishes if $X$ has no $K$-rational point.
From the key lemma we deduce the following theorem.

\begin{thm*} \emph{(Theorem \ref{tolle idee})}
Let $X$ be a smooth, proper $K$-variety.
Let $L/K$ be a tame Galois extension, $X_L$ the base change of $X$ to $L$.
Let $\mathcal{Y}$ be a weak N\'eron model of ${X}_L$ with a good
$G\!:=\Gal(L/K)$-action extending the Galois action on $X_L$. Then
\[
S(X)=[\mathcal{Y}^G]\in K_0^{\mathcal{O}_K}(\Var_k)/(\mathbb{L}-1).
\]
\end{thm*}

The rational volume $s(X)$ of a $K$-variety $X$
is defined to be the Euler characteristic with proper support and coefficients in $\mathbb{Q}_l$, $l\neq \Char (k)$ a prime, of the special fiber of a weak N\'eron model of $X$.
The rational volume vanishes if $X$ has no $K$-rational point, too.

\begin{thm*} \emph{(Theorem \ref{rational volume})}
Let $X$ be a smooth, proper $K$-variety, and let $L/K$ be a tame Galois extension of degree $q^r$, $q$ a prime.
Then $s(X)=s(X_L) \mod q$.
\end{thm*}

The proof of this theorem uses the fact that there is always a weak N\'eron model of $X_L$ with an action of $\Gal(L/K)$ extending the Galois action on $X_L$, see Theorem~\ref{gaction on wnm}, as well as the equation for the Serre invariant (Theorem~\ref{tolle idee}).
Moreover, we use the fact that for a scheme of finite type $V$ over some field with a good action of a $q$-group $G$, we have $\chi_c(V)=\chi_c(V^G)\mod q$.
This argument goes back to \cite[Section 7.2]{MR2555994}.\\
Finally, we can deduce the existence of rational points for some varieties with potential good reduction.
By definition, a $K$-variety $X$ has potential good reduction if there is a Galois extension $L/K$ such that
the base change of $X$ to $L$ admits a smooth and proper model.

\begin{cor*} \emph{(Corollary \ref{korollar rational volume})}
Let $X$ be a smooth, proper $K$-variety with potential good reduction after a base change of order $q^r$,
$q\neq \Char(k)$ a prime.
If the Euler characteristic of $X$ with coefficients in $\mathbb{Q}_l$, $l\neq \Char(k)$ a prime, does not vanish modulo $q$,
then $X$ has a $K$-rational point.
\end{cor*}

To prove this corollary we use Theorem \ref{rational volume}, and the fact that the Euler characteristic with coefficients in $\mathbb{Q}_l$ is constant on the fibers of a smooth and proper morphism.\\
In addition, we obtain a similar result for the Euler characteristic with coefficients in the structure sheaf,
see Corollary \ref{Korollar}.
In this Corollary we need to assume that
there is a tame Galois extension $L/K$ of prime degree,
such that there is a smooth and proper model of $X_L$ with a good $G$-action extending the Galois action on $X_L$, 
because we cannot use the results concerning the motivic invariants.
We show directly that the $G$-action on this smooth and proper model of $X_L$ has a closed fixed point, and
use  Corollary \ref{sections of the quotient} to conclude that in this case the model obtained by taking the quotient will have an $\mathcal{O}_K$-point inducing a $K$-point of $X$.

\subsubsection*{Acknowledgements}

The results contained in this article are part of the author's dissertation, written under the supervision of  H\'el\`{e}ne Esnault.
My thesis was supported by the SFB/TR45 "Periods, moduli spaces and arithmetic of algebraic varieties" of the DFG (German Research Foundation).
I would like to thank the members of the "Essener Seminar f\"ur Algebraische Geometrie und Arithmetik'' for their support.
In particular, I am very thankful to Andre Chatzistamatiou for the numerous discussions we had, and for all the suggestions he made. 
I thank Johannes Nicaise and Olivier Wittenberg for reading my thesis and for helping me with important remarks.
I am very grateful to H\'el\`{e}ne Esnault for the time and ideas she gave me and for her constant support.

\subsubsection*{Conventions}

A variety over a field $F$ is a geometrically integral, separated $F$-scheme of finite type over $F$.
We assume that an integral scheme is connected.
All schemes are assumed to be noetherian.\\
If $U$ is a $V$-scheme, $\Spec(F)\to V$ any point. We set $U_F:=U\times_V\Spec(F)$.

\medskip

In the entire article, let 
K be a complete local field with ring of integers $\mathcal{O}_K$, $S\!:=\Spec(\mathcal{O}_K)$,
and residue field $k$.
Assume that $k$ is algebraically closed.

\section{Models with Galois Actions}
\begin{defn}
Let $X$ be a $K$-variety.
A \emph{model} of $X$ is an integral $S$-scheme $\mathcal{X}$ of finite type over $S$ such that $\mathcal{X}_K\cong X$.
\end{defn}

\begin{rem}\label{flach}
Let $X$ be a non-empty $K$-variety, and let $\mathcal{X}\to S$ be any model of $X$.
Then $\mathcal{X}$ dominates $S$, so by \cite[Chapter III, Proposition 9.7]{MR0463157} $\mathcal{X}$ is flat over $S$.
\end{rem}

\begin{rem}\label{points and models}
Let $\varphi: \mathcal{X}\to S$ be a model of a $K$-variety $X$.
Then we have maps as follows induced by the universal property of the fiber product.
\[
 \xymatrix{
X(K) & \mathcal{X}(\mathcal{O}_K) \ar[r]^s \ar[l] & \mathcal{X}_k(k)
} 
\]
If $\varphi$ is proper, $\mathcal{X}(\mathcal{O}_K)\to X(K)$ is bijective by the valuative criterion of properness.
If $\varphi$ is smooth, the specialization map $s$ is surjective by \cite[Chapter~2.3, Proposition~5]{MR1045822}, because $\mathcal{O}_K$ is Henselian.
 \end{rem}

\begin{defn}
A \emph{weak N\'eron model} of a smooth $K$-variety $X$ is a smooth and separated model $\mathcal{X}\to S$ of $X$,
such that the natural map $\mathcal{X}(\mathcal{O}_K)\to X(K)$ is a bijection.
\end{defn}

\begin{rem}
Let $X$ be a smooth $K$-variety attached with a weak N\'eron model $\mathcal{X}\to S$.
Then $X(K)=\emptyset$ if and only if the special fiber $\mathcal{X}_k$ of $\mathcal{X}\to S$ is empty.
This is true, because by definition the natural map $\mathcal{X}(\mathcal{O}_K)\to X(K)$ is a bijection,
the specializing map $\mathcal{X}(\mathcal{O}_K)\to \mathcal{X}_k(k)$ is surjective by Remark \ref{points and models},
and $k$ is algebraically closed.
\end{rem}

\begin{rem}
A weak N\'eron model does not exist for all smooth $K$-varieties $X$.
It follows from \cite[Chapter 3.5, Theorem 2]{MR1045822} that a weak N\'eron model exists if $X$ is proper over $K$.

Note that a weak N\'eron model is not unique. Take any weak N\'eron model, blow up a point in the special fiber, and then take the smooth locus of the obtained scheme.
This is again a weak N\'eron model.
\end{rem}

Now fix a Galois extension $L/K$ with Galois group $G:=\Gal(L/K)$.
Let $\mathcal{O}_L$ be the ring of integers of $L$, $T\!:=\Spec(\mathcal{O}_L)$.
Note that $k$ is the residue field of $L$.
For a general introduction to local fields and their Galois extensions we refer to
\cite{MR554237}. 

A Galois extension $L/K$ is called  \emph{tame}, if the order of its Galois group is prime to $\Char(k)$.
From
\cite[Chapter IV, Corollary 2 and Corollary 4]{MR554237} we get that the Galois group of a tame Galois extension $L/K$ is always cyclic.

\medskip

We now want to consider group actions of the Galois group.
Therefore, recall the following facts concerning group actions of an abstract finite group $G$.

\medskip

Let  $U$ be a scheme, $\Aut(U)$ the abstract group of automorphisms of $U$.
A \emph{$G$-action} on $U$ is given by a group homomorphism $\mu_U:G\to \Aut(U)$.
If $U$ is an affine scheme, i.e.~$U=\Spec(A)$, then a group action on $U$ is also given by a group homomorphism ${\mu_U^\#:G\to \Aut(A)}$.

Let $U$, $V$ be schemes with $G$-actions.
We call a morphism of schemes $f : U \to V$ \emph{$G$-equivariant}, if for all $g\in G$ we have $f \circ \mu_U(g)=\mu_V(g)\circ f$.
\medskip

A $G$-action on a scheme $U$ is called \emph{good}, if every orbit is contained in an affine open subscheme of $U$.
By \cite[Expos\'e V, Proposition 1.8]{MR0217087} this is the same as requiring a cover of $U$ by affine, open, $G$-invariant subschemes.

If $U$ is a scheme with a good $G$-action,
then there exists a quotient $\pi: U\to U/G$ in the category of schemes, see \cite[Expos\'e V.1]{MR0217087}.

\begin{rem}\label{Galois action}
By definition of the Galois group, $G$ acts on $L$, and $K=L^{G}$.
The $G$-action of $L$ can be restricted to $\mathcal{O}_L$, and $\mathcal{O}_L^{G}=\mathcal{O}_K$.
We call this action the \emph{Galois action} on $\mathcal{O}_L$.
Note that $\Spec(L)\hookrightarrow T$ is ${G}$-equivariant for these actions.
Let $X$ be a $K$-variety. 
As $X$ is flat over $K$, by \cite[Expos\'e V, Proposition~1.9]{MR0217087}, $G$ acts on $X_L$
such that $X_L\to \Spec(L)$ is $G$-invariant and $X_L/G\cong X$.
We call this action the \emph{Galois action} on $X_L$.
\end{rem}

\begin{rem}
\label{XL to X}
Let $X$ be a $K$-variety.
Let $\varphi: \mathcal{Y}\to T$ be a model of $X_L$ with a good ${G}$-action.
Assume that  ${X_L\hookrightarrow \mathcal{Y}}$ is $G$-equivariant for the action on $\mathcal{Y}$ and the Galois action on $X_L$,
i.e.~the $G$-action on $\mathcal{Y}$ extends the Galois action on $X_L$.\\
Take any $h\in G$,
and let $g\in \Aut(\mathcal{Y})$ and $g_T\in \Aut(T)$ be its images.
As the maps $X_L\hookrightarrow \mathcal{Y}$, $X_L\to \Spec(L)$, and $\Spec(L)\hookrightarrow T$ are ${\Gal(L/K)}$-equivariant,
we obtain that $g_T \circ \varphi \circ g^{-1}\!\mid _{X_L}=\varphi\!\mid_{X_L}$.
As $X_L\subset \mathcal{Y}$ is open and dense, $\mathcal{Y}$ is reduced, and $T$ is separated,
\cite[Corollary 9.9]{ MR2675155} implies that $g_T \circ \varphi \circ g^{-1}=\varphi$, i.e.~$\varphi$ is $G$-equivariant.\\
Let $\pi:\mathcal{Y}\to \mathcal{X}:=\mathcal{Y}/G$ be the quotient.
Using that the maps in the square on the left hand side are $G$-equivariant, we get the following big commutative diagram.
\[
  \xymatrix{
X_L\ar[d]\ar[r] \ar@/^6mm/[rrr] & \mathcal{Y}\ar[d]^\varphi \ar@/^6mm/[rrr]^\pi\ar@{}[ld]|\square & & X\ar[d]\ar[r] & \mathcal{X}\ar[d]\\
\Spec(L)\ar[r]\ar@/_6mm/[rrr] & T \ar@/_6mm/[rrr]_{}& & \Spec(K)\ar[r] & S
}
\]
Note that $\mathcal{X}$ is an $S$-scheme of finite type by \cite[Expos\'e V, Proposition~1.5]{MR0217087}.
As $\mathcal{X}$ is a quotient by a finite group of the integral scheme $\mathcal{Y}$, it is integral, too.
As $\Spec(L)\hookrightarrow T$ is flat, by \cite[Expos\'e V, Proposition 1.9]{MR0217087} we obtain
\[
\mathcal{X}_K=\mathcal{Y}/G\times_S \Spec(K) \cong \mathcal{Y}\times_S \Spec(K)/{G}=X_L/{G}=X.
\]
Hence $\mathcal{X}\to S$ is a model of $X$. 
\end{rem}

In general the quotient $\mathcal{X}$ will be singular. To see this, look at the following example.

 \begin{ex}
 \label{the example}
Let $k$ be an algebraically closed field with $\Char (k)\neq 2$, and
set ${K\!:=k\Mel s\Mer}$, $L\!:=k\Mel t \Mer$ with $t^2 =s$.
Hence $L/K $ is a tame Galois extension with Galois group $G=\mathbb{Z}/2\mathbb{Z}$.
The Galois action on $k\Mel t\Mer$ is given by
\[
\alpha:k\Mel t\Mer \to k\Mel t\Mer ;\  P(t)\mapsto P(-t).
\]
Set $X:=\mathbb{A}^1_K=\Spec(k\Mel s \Mer[y])$.
The Galois action on $X_L=\Spec(k\Mel t \Mer[y])$ is given by
\[
\beta :k\Mel t\Mer[y] \to k\Mel t\Mer[y] ;\  P(t,y)\mapsto P(-t,y).
\]
Look at the smooth $\mathcal{O}_L=k\Pol t\Por$-scheme $\mathcal{Y}:=\mathbb{A}^1_{k\Pol t \Por}=\Spec(k[\![t]\!][x])$ with the $G$-action given by
\[
\gamma :k\Pol t\Por[x]\to k\Pol t\Por[x];\ P(t,x)\mapsto P(-t,-x).
\]
Using the fact that $t$ is invertible in $k\Mel t \Mer$, one shows that the map
\[
X_L=\Spec(k\Mel t \Mer(y))\to \mathcal{Y}_L=\Spec(k\Mel t \Mer(x))
\]
given by sending $t$ to $t$ and $y$ to $xt$
is a $G$-equivariant isomorphism.
Hence $\mathcal{Y}$ is a model of $X_L$ with a $G$-action extending the Galois action on $X_L$.

The $k\Pol t \Por^G\!\!=\!k \Pol t^2 \Por=k\Pol s \Por=\mathcal{O}_K$-scheme
\[
\mathcal{X}:=\mathcal{Y}/G=\Spec(k\Pol t \Por[x]^G)= \Spec(k\Pol t^2\Por [tx,x^2]) \cong {\Spec(k\Pol s\Por [b,c]/(sc-b^2))}
\]
is singular in $(0,0,0)$,
so by Remark \ref{XL to X} it $\mathcal{X}$ is a singular model of $X$.
\end{ex}

\begin{rem}
 In order to get a projective example, replace in Example~\ref{the example}
 $X=\mathbb{A}^1_K$ by $\mathbb{P}^1_K$ and $\mathcal{Y}=\mathbb{A}^1_{k\Pol t \Por}$ by
$\mathbb{P}_{k\Pol t \Por}^1$ with a $G$-action given by $g\in \Aut(\mathbb{P}_{k\Pol t \Por}^1)$
such that 
\[
g(t,[x_0:x_1])=(-t,[-x_0:x_1]).
\]
\end{rem}

For a given $K$-variety $X$ and a Galois extension $L/K$, there exist interesting models of $X_L$ with a good action of the Galois group as in Remark \ref{XL to X}.
In this context we can show the following theorem.

\begin{thm} \label{gaction on wnm}
Let $L/K$ be a Galois extension with Galois group $G\!:=\Gal(L/K)$, and let $\mathcal{O}_L$ be the ring of integers of $L$, $T\!:=\Spec(\mathcal{O}_L)$.
Then for a given smooth and proper $K$-variety $X$,
there exists a weak N\'eron model 
$\varphi:\mathcal{Y}\to T$
of ${X_L}$ and a good $G$-action on $\mathcal{Y}$ extending the Galois action on $X_L$.
\end{thm}

\begin{proof}
In order to prove this theorem, we need to recall how a weak N\'eron model is constructed.
The main tool of showing that weak N\'eron models actually exist is the so called N\'eron smoothening.

\begin{defn}
Let $X$ be a smooth $K$-variety, and let $\mathcal{X}\to S$ be a model of $X$.
A \emph{N\'eron smoothening} of $\mathcal{X}$ is a proper $S$-morphism $f\!:\mathcal{X}'\to \mathcal{X}$
such that $f$ is an isomorphism on the generic fibers,
and such that the canonical map $\Sm(\mathcal{X}'/S)(S)\to \mathcal{X}'(S)$ is bijective.
Here $\Sm(\mathcal{X}'/S)$ is the smooth locus of $\mathcal{X}'$ over $S$.
\end{defn}

In order to prove Theorem \ref{gaction on wnm} we need the following Lemma.

\begin{lem}\label{blowup ginvariant}
Let $Y$ be a smooth $L$-variety, let
$\mathcal{Y}\to T$ be a model of $Y$ with a good $G$-action, and assume
that the structure map $\varphi: \mathcal{Y} \to T$ is $G$-equivariant for this action and the Galois action on $T$.
Then there exists a projective N\'eron smoothening $f: \mathcal{Y}' \to \mathcal{Y}$, and a good $G$-action on $\mathcal{Y}'$ such that $f$ is $G$-equivariant.  
\end{lem}

\begin{proof}
By \cite[Chapter 3.1, Theorem 3]{MR1045822}
there exists a N\'eron smoothening $f:\mathcal{Y}'\to \mathcal{Y}$,
which consists of a finite sequence of blowups
with centers in the special fibers.
We need to construct a $G$-action on $\mathcal{Y}'$ such that $f$ is $G$-equivariant.\\
Note that if we blow up an integral scheme $U$ with a good $G$-action in a closed $G$-invariant subscheme $V\subset U$, and denote by $u:U' \to U$ the blowup,
then there is a $G$-action on $U'$ such that $u$ is $G$-equivariant.
The reason for this is the following.
Let $h\in G$ be any element, $g_U\in \Aut(U)$ its image. As $V$ is $G$-invariant, $g_U(V)=V$.
So by the universal property of blowup, see \cite[Chapter II, Corollary 7.15]{MR0463157}, there exists a unique $g_{U'}\in \Aut(U')$ such that $u\circ g_{U'}=g_U \circ u$.
This way we can define the required group action on $U'$, and $u$ is $G$-equivariant by construction.

Consider $f$, which is a sequence of blowups, i.\thinspace e.~we have
\[
  \xymatrix{
\mathcal{Y}'=:\!\mathcal{Y}_m\ar[r]^{f_{m-1}}\ar@/_6,6mm/[rrrr]_f&\mathcal{Y}_{m-1}\ar[r]^{f_{m-2}}&\dots \ar[r]^{f_{1}}&\mathcal{Y}_1\ar[r]^/-2mm/{f_{0}}&\mathcal{Y}_0\!:=\mathcal{Y}.
}
\]
Here the $f_i$ are blowups of some closed subschemes $V_i\subset \mathcal{Y}_i$.
One checks in the proof of \cite[Chapter 3.4, Theorem 2]{MR1045822}
that all the $V_i$ are obtained using the same construction.
Hence if we show that $V\!:=V_0\subset \mathcal{Y}$ is $G$-invariant,
then we obtain a $G$-action on $\mathcal{Y}_1$ such that $f_0$ is $G$-equivariant,
hence $\varphi \circ f_0$ is $G$-equivariant,
and we can conclude inductively on the length of the sequence of blowups.

One can check in \cite[Chapter 3.4, Theorem 2]{MR1045822} that
$V$ is constructed as follows.
Let ${E\subset \mathcal{Y}(\mathcal{O}_L)}$ be the subset of all $\sigma \in \mathcal{Y}(\mathcal{O}_L)$ not factoring through $\Sm(\mathcal{Y}/T)$,
with $\Sm(\mathcal{Y}/T)$ the smooth locus of $\mathcal{Y}$ over $T$.
Set $F^1:=E$.
Let  $s:\mathcal{Y}(\mathcal{O}_L)\to \mathcal{Y}_k(k)$ be the specialization map, and
$V^i$ the Zariski closure of $s(F^i)$ in $\mathcal{Y}$, and
let $U^i\subset V^i$ be the largest open subset, such that $U^i$ is smooth over $k$, and that $\Omega^1_{\mathcal{Y}/T}\!\!\mid_{V^i}$ is locally free over $U^i$.
Set $E^i:= \{ a\in F^i\mid s(a)\in U^i\}$, and
$F^{i+1}:=F^i\!\setminus\! E^i$.
Note that there is a minimal $t\in \mathbb{N}$ such that $F^{t+1}= \emptyset$.
Set $V=V^t$.

The action of $G$ on $\mathcal{Y}$ induces a $G$-action on $\mathcal{Y}(\mathcal{O}_L)$, and, as $\varphi$ is $G$-equivariant, and hence $\mathcal{Y}_k\subset \mathcal{Y}$ is $G$-invariant, a $G$-action on $\mathcal{Y}_k(k)$.
Note that $s$ is $G$-equivariant.
We now show by induction that $F^i$ is $G$-invariant for all $i$.

Take any $h\in G$, let $g$ be its image in $\Aut(\mathcal{Y})$, $g_T$ its image in $\Aut(T)$.
Note that
\[
 \varphi\! \mid _{g(\Sm(\mathcal{Y}/T))}= {g_T}\circ \varphi \circ g^{-1}\!\! \mid_{g(\Sm(\mathcal{Y}/T))}= {g_T}\circ \varphi\!\mid_{\Sm(\mathcal{Y}/T)} \circ g^{-1}\!\!\mid_{g(\Sm(\mathcal{Y}/T))}.
\]
% The second equation holds, because $g^{-1}$ maps $g(\Sm(\mathcal{Y}/T))$ to $\Sm(\mathcal{Y}/T)$.
Hence $\varphi\mid_{g(\Sm(\mathcal{Y}/T))}$  is smooth, which implies that $\Sm(\mathcal{Y}/T)$ is $G$-invariant, hence $E=F^1$ is $G$-invariant.
So we may assume that $F^i$ is $G$-invariant for some $i$.

Consider $Z_i\!:=\cap_{h\in G}h(V^i)\subset V^i$.
By construction, $Z_i$ is closed in $\mathcal{Y}$ and $G$-invariant.
As $F^i$ is $G$-invariant by assumption, $s(F^i)\subset V^i$ is $G$-invariant, and hence $s(F^i)\subset h(V^i)$ for all $h\in G$, i.e.~$s(F^i)\subset Z_i$.
So by definition of the Zariski closure, $V^i=Z_i$, hence in particular $V^i$ is $G$-invariant.\\
Let $\Sm(V^i)$ be the smooth locus of $V^i$ over $k$.
Note that $U^i=\Sm(V^i)\cap W^i$, with $W^i\subset V^i$ the largest open subset over which $\Omega^1_{\mathcal{Y}/T}\!\!\mid_{V^i}$  is locally free.
As the $G$-action on $V^i$ is given by isomorphisms, regular points are mapped to regular points, hence $\Sm(V^i)$ is $G$-invariant. Now we examine $W^i$.
Since $G$ acts equivariantly on $\mathcal{Y}\to T$, the natural map
$g^*(\Omega^1_{\mathcal{Y}/T}) \to \Omega^1_{\mathcal{Y}/T}$ is an isomorphism.
As $V^i$ is $G$-invariant,
$g^*(\Omega^1_{\mathcal{Y}/T})\!\!\mid_{V^i} \to \Omega^1_{\mathcal{Y}/T}\!\!\mid_{V^i}$
is an isomorphism, too.
Altogether we obtain
\[
\Omega^1_{\mathcal{Y}/T}\!\!\mid_{V^i\cap W^i}\cong g^*(\Omega^1_{\mathcal{Y}/T})\!\!\mid_{V^i\cap W^i}=g^*(\Omega^1_{\mathcal{Y}/T}\!\!\mid_{V^i\cap g^{-1}(W^i)}).
\]
As the first is locally free by definition of $W^i$, $g^*(\Omega^1_{\mathcal{Y}/T}\!\!\mid_{V^i\cap g^{-1}(W^i)})$ is locally free, too.
As $g$ is an automorphism of $\mathcal{Y}$, $\Omega^1_{\mathcal{Y}/T}\!\!\mid_{V^i\cap g^{-1}(W^i)}$ is locally free.
Hence by definition of $W^i$, $g^{-1}(W^i)\subset W^i$,
i.\thinspace e.~$W^i$ is $G$-invariant.
Hence $U^i=\Sm(V^i)\cap W^i$ is $G$-invariant, too.

Using that $s$ is $G$-equivariant, we get that $E^i$ is $G$-invariant, and therefore $F^{i+1}$ is $G$-invariant, too.
So it follows by induction that for all $i$, $F^i$ is $G$-invariant, in particular $F^t$ is $G$-invariant.
Using the same argument as in the induction, we can show that $V^t=V$ is $G$-invariant, and this is what we wanted to show.

We still need to show that the $G$-action on $\mathcal{Y}$ is good.
So take any orbit in $\mathcal{Y}'$. Its image under $f$ will be contained in an open affine subset $U\subset \mathcal{Y}$.
As $f$ is projective, $f^{-1}(U)$ is projective over $U$ and contains our orbit, which is finite, because $G$ is finite.
By \cite[Chapter~3, Proposition~3.36.b]{ MR1917232}
there is an affine subset $U'\subset f^{-1}(U)$ containing every finite set of points.
Hence the action on $\mathcal{Y}'$ is good.
\end{proof}

In \cite{private} the following similar theorem in the context of formal schemes is proven:

\begin{thm*}
Any generically smooth, flat, separated formal $\mathcal{O}_L$-scheme $X_\infty$, topologically of finite type over $\mathcal{O}_L$,
endowed with a good $G$-action compatible with the $G$-action on $\mathcal{O}_L$, admits a $G$-equivariant N\'eron smoothening.
\end{thm*}

\medskip

Now we are finally ready to prove Theorem \ref{gaction on wnm}.
So let $X$ be a smooth, proper $K$-variety.
In particular $X$ is a separated $S$-scheme of finite type, hence
by Nagata's embedding Theorem, see \cite[Theorem 12.70]{ MR2675155}, there exists a proper, integral $S$-scheme $\mathcal{X}$
and an immersion $X\hookrightarrow \mathcal{X}$ over $S$ which is schematically dense.
As $X$ is proper over $K$, $\mathcal{X}_K$ is in fact isomorphic to $X$.
Altogether, $\mathcal{X}\to S$ is a proper model of $X$.

Set $\mathcal{X}_T\!:=\mathcal{X}\times_S T$, and
let $\Phi: \mathcal{X}_T\to T$ be the projection to $T$.
Note that $\Phi$ is proper, and
$\mathcal{X}_T\times_T\Spec(L)=X_L$.
By Remark~\ref{flach}, $\mathcal{X}$ is flat over $S$, therefore $\mathcal{X}_T$ is flat over $T$.
Hence there cannot be a connected component of $\mathcal{X}_T$ only supported on the special fiber. But the generic fiber $X_L$ of $\mathcal{X}_T$ is connected, hence $\mathcal{X}_T$ is connected.
Hence one can check locally that
$\mathcal{X}_T$ is integral, which is straightforward to check.
Altogether, $\Phi: \mathcal{X}_T \to T$ is a proper model of X.

As $\mathcal{X}\to S$ is flat, by \cite[Expos\'e V, Proposition 1.9]{MR0217087} there exists a good $G$-action on $\mathcal{X}_T$
such that $\Phi$ is $G$-equivariant, and $\mathcal{X}_T/G\cong \mathcal{X}$.
This $G$-action extends the Galois action on $X_L$ by construction.

By Lemma \ref{blowup ginvariant} there exists a projective N\'eron smoothening
$f:\mathcal{Y}'\to \mathcal{X}_T$,
and a good $G$-action on $\mathcal{Y}'$ such that $f$ is $G$-equivariant.
Let $\mathcal{Y}\subset \mathcal{Y}'$ be the smooth locus of $\Phi \circ f$.
Set $\varphi:=\Phi\circ f\mid_{\mathcal{Y}}$.
Note that $\varphi$ is separated. We have the following commutative diagram.
\[
  \xymatrix{
\mathcal{Y} \ar[ddr]_{\varphi}\ar@{^(->}[r] & \mathcal{Y}' \ar[d]^{f}\\
& \mathcal{X}_T \ar[d]^{\Phi}\\
 & T
}
\]
Note that $X_L$ is smooth, because $X$ is smooth.
As $f$ is a N\'eron smoothening, $\mathcal{Y}'_L=\mathcal{X}_T\times_T \Spec(L)=X_L$.
Hence $\mathcal{Y}'_L$ is in particular smooth over $T$, so
$\mathcal{Y}_L=\mathcal{Y}'_L=X_L$.
As $\mathcal{X}_T$ is integral, $\mathcal{Y}'$ and $\mathcal{Y}$ are integral, too.
Hence $\varphi: \mathcal{Y}\to T$ is a smooth and separated model of $X_L$.
As $\Phi$ and $f$ are proper, by the valuative criterion of properness the natural map
$\mathcal{Y}'(\mathcal{O}_L)\to X_L(L)$ is a bijection.
As $f$ is a N\'eron smoothing, $\mathcal{Y}'(\mathcal{O}_L)=\mathcal{Y}(\mathcal{O}_L)$. 
So $\varphi: \mathcal{Y}\to T$ is a weak N\'eron model of $X_L$.

We still need to show that there is a good $G$-action on $\mathcal{Y}$ extending the Galois action on $X_L$.
As $f$ is $G$-invariant for the $G$-action on $\mathcal{Y}'$ and $\mathcal{X}_T$, the $G$-action $\mathcal{Y}'$ extends the Galois action on $X_L$.
So it suffices to show that this $G$-action restricts to $\mathcal{Y}$,
i.e.~that $\mathcal{Y}\subset \mathcal{Y}'$ is $G$-invariant. 
To show this, we can simply use the proof in the base case of the induction in Lemma \ref{blowup ginvariant}.
Note that the action is good for the following reason.
Take any orbit in $\mathcal{Y}$. As the action on $\mathcal{Y}'$ is good, it is contained an open affine subset $U\subset \mathcal{Y}'$.
So it is contained in $U\cap \mathcal{Y}$, which is open in $U$. So by \cite[Chapter~3, Proposition~3.36.b]{ MR1917232}
there is an affine subset $U'\subset U\cap \mathcal{Y}$ containing our finite orbit.
\end{proof}

In \cite[Proposition 4.5]{1009.1281} the following similar statement is proven.

\begin{prop*}
Let $G$ be any finite group, $X$ a smooth and proper $K$-variety,
endowed with a good $G$-action.
Then there is a weak N\'eron model $\mathcal{X}\to S$ of $X$ endowed with a good $G$-action,
such that $X\hookrightarrow \mathcal{X}$ is $G$-equivariant.
\end{prop*}

\section{A Canonical Weak N\'eron Model of a Quotient Scheme}
\begin{thm}
\label{main theorem}
Let $L/K$ be a tame Galois extension, $G:=\Gal(L/K)$.
Let $\mathcal{O}_L$ be the ring of integers of $L$, $T:=\Spec(\mathcal{O}_L)$.
Let $X$ be a smooth $K$-variety, and let $\varphi: \mathcal{Y}\to T$ be a smooth model of $X_L$ with a good $G$-action extending the Galois action on $X_L$.
Let $\mathcal{X}:=\mathcal{Y}/G$ be the quotient.\\
Then there is a smooth model $\mathcal{Z}\to S$ of $X$ and a separated $S$-morphism ${\Phi: \mathcal{Z}\to \mathcal{X}}$,
such that the induced map $\mathcal{Z}(S)\to \mathcal{X}(S)$ is a bijection, and such that
for all smooth integral $S$-schemes $\mathcal{V}$ and
all dominant $S$-morphisms $\Psi: \mathcal{V}\to \mathcal{X}$ there is a unique S-morphism $\Psi':\mathcal{V}\to \mathcal{Z}$ making the following diagram commutative.
\[
  \xymatrix{
 \mathcal{V}\ar[rd]^{\Psi'} \ar[dd]_\Psi\\
 & \mathcal{Z} \ar[dl]^\Phi  \\
\mathcal{X}
 }
\]
In particular $\mathcal{Z}$ is unique with this property up to a unique isomorphism.\\
If $\mathcal{Y}\to T$
is a weak N\'eron model of $X_L$, then $\mathcal{Z}\to S$ is a weak N\'eron model of $X$.
\end{thm}

\begin{proof}
The proof consists of six steps. 
First we will give the construction of $\mathcal{Z}$ as a functor of schemes,
then we construct $\Phi$ as a morphism of functors.
In the third step we will show that $\mathcal{Z}$ is represented by a smooth $S$-scheme.
Thereafter we show the properties of $\Phi$, namely that it is separated and that the map $\mathcal{Z}(S)\to \mathcal{X}(S)$ induced by $\Phi$ is an isomorphism.
Afterwards we show the universal property.
In the final step we consider the case that $\mathcal{Y}$ is a weak N\'eron model of $X_L$.

\medskip
 
\textit{Construction of $\mathcal{Z}$.}
We now construct $\mathcal{Z}$.
The construction can be found in 
\cite{ MR1149171}, where it is used in the context of abelian varieties.

\begin{defn}
The \emph{Weil restriction} of a $T$-scheme $U$ to $S$ is defined as the functor
\begin{align*}
 \Res_{T/S}(U): \ &(\Sch/S)\to (\Sets)\\
 &W\mapsto \Hom _T(W\times _ST,U).
\end{align*}
\end{defn}

\begin{defn}\label{y^G}
Let $V$ be an $S$-scheme with a $G$-action, such that the structure map is $G$-equivariant for this action and the trivial action on $S$.
We define the \emph{functor of fixed points} by
\begin{align*}
V^G: \ &(\Sch/S)\to (\Sets)\\
 &W\mapsto V(W)^G=\Hom _S(W,V)^G.
\end{align*}
By \cite[Proposition 3.1]{ MR1149171}
this functor is represented by a subscheme of $V$.
We call this scheme the \emph{fixed locus} of the $G$-action on $V$.
\end{defn}

Note that $G$ is a finite cyclic group, because $L/K$ is a tame Galois extension.
Therefore every $G$-action is given by one automorphism.

\begin{con} \label{construction Z} \cite[Construction 2.4 and Theorem 4.2]{ MR1149171}\\
Fix a generator of $G$, and let $g\in \Aut (\mathcal{Y})$ and $g_T\in \Aut(T)$ be its images.
Then $\tilde{g}\in \Aut(\Res_{T/S}(\mathcal{Y}))$, which maps $f\in \Hom  _T(W\times _ST,\mathcal{Y})$ to $g\circ f \circ (\Id_W\times g_T)^{-1}$
for every $W\in (\Sch/S)$, defines a $G$-action on $\Res_{T/S}(\mathcal{Y})$.
It is easy to see that $\tilde{g}$ is an $S$-morphism. Therefore the structure map $\Res_{T/S}(\mathcal{Y})\to S$ is $G$-equivariant for the $G$-action on $\Res_{T/S}(\mathcal{Y})$ and the trivial $G$-action on $S$.
Define
\begin{align*}
 \mathcal{Z}: \  &(\Sch/S)\to (\Sets)\\
 &W\mapsto (\Res_{T/S}(\mathcal{Y}))^G (W) = \Hom_T(W\times_S T, \mathcal{Y})^G.
\end{align*}
\end{con}

Note that $\Hom_T(W\times_ST,\mathcal{Y})^G$ is the set of $G$-equivariant $T$-morphisms from $W\times_ST$ to $\mathcal{Y}$.

\medskip

\textit{Construction of $\Phi$.}
View $\mathcal{X}$ and $\mathcal{Z}$ as functors from the category of flat $S$-schemes to the category of sets.
We now construct a morphism of functors $\Phi:\mathcal{Z}\to \mathcal{X}$.
As soon as we will have shown that $\mathcal{Z}$ is in fact representable by a flat $S$-scheme,
this will yield an $S$-morphism of schemes by Yoneda's lemma for the category of flat $S$-schemes.

We need to construct maps 
$\Phi(W) : \mathcal{Z}(W)=\Hom _T(W\times_ST,\mathcal{Y})^G\to \mathcal{X}(W)$
for all flat $W\in (\Sch/S)$, and show that they are functorial.
Take any $f\in \mathcal{Z}(W)$.
Let $\pi: \mathcal{Y}\to \mathcal{X}$ be the quotient map.
We have the following commutative diagram.
\begin{align*} 
  \xymatrix{
  & \mathcal{Y} \ar[rd]^{\pi} \ar[d]^{\varphi}  \\
{W\times_ST} \ar[ur]^f\ar[r]\ar[d]_{p_W} & T \ar[d] & {\mathcal{X}} \ar[dl] \\
W \ar@{-->}[urr]^{f'} \ar[r] & S 
}
\end{align*}

As $W\to S$ is flat, \cite[Expos\'e V, Proposition 1.9]{MR0217087} implies that the projection map $p_W: W\times _ST \to W$ is the quotient of the $G$-action on $W\times_ST$ given by $\Id_W\times g_T$.
As $f=g\circ f \circ (\Id_W\times g_T)^{-1}$, and $\pi$ is $G$-equivariant for the $G$-action on $\mathcal{Y}$ and the trivial action on $\mathcal{X}$, we get
$(\pi \circ f) \circ (\Id_W\times g_T)= \pi \circ g \circ f= \pi \circ f$.
Hence $\pi \circ f$ is $G$-equivariant for the $G$-action on $W\times_S T$ and the trivial action on $\mathcal{X}$, and
therefore, by the universal property of the quotient ${p_W:W\times _ST \to W}$ we obtain a unique $f'\in \mathcal{X}(W)$ making the diagram above commutative.
We set $\Phi(W)(f)\!:=f'$.
It is easy to check that this map is functorial.

\medskip

View $\mathcal{X}_K\cong X$ as a presheaf on the category of $K$-schemes.
We now construct an inverse map of functors ${{\Phi\!\mid_{ \mathcal{Z}_K}}^{-1}:X\to \mathcal{Z}_K\subset \mathcal{Z}}$.
Take any $W\in(\Sch/K)$.
Note that $W\times_S T\cong W_L$, hence
\begin{align*}
 \mathcal{Z}_K(W)&=\Hom_T(W\times_ST,\mathcal{Y})^G
=\Hom_T(W_L,\mathcal{Y})^G
=\Hom_L(W_L,X_L)^G.
\end{align*}
Take any $h \in X(W)$, and consider the following diagram with $p_L$ and $p_W$ the projection maps.
\[
  \xymatrix{
W _L\ar@/_1,7pc/[ddr]_{p_L} \ar@{-->}[rd]^{h^*} \ar@/^1,7pc/[drr]^{h\circ p_W} \\
& {X_L} \ar[r]^{\!\!\pi\mid_{X_L}}\ar[d]_{\varphi\mid_{X_L}} & {X} \ar[d] \\
& \Spec (L) \ar@{}[ur]|{\square} \ar[r] & \Spec (K) 
}
\]
As $h$ is a $K$-morphism, the diagram commutes, and hence the universal property of fiber product induces a unique $h^*\in \Hom_L(W_L,X_L)$ with $\pi\circ h^*=h\circ p_W$.
Using that $h^*$ is unique, one can easily show that it is actually $G$-equivariant, hence we may set
 ${\Phi\!\mid_{ \mathcal{Z}_K}}^{-1}(W)(h)\!:=h^*$.
It is straightforward to check functoriality and the fact that ${\Phi\!\mid_{ \mathcal{Z}_K}}^{-1}\circ \Phi\!\mid_{\mathcal{Z}_K}=\Id_{\mathcal{Z}_K}$, and $\Phi\circ {\Phi\!\mid_{ \mathcal{Z}_K}}^{-1}=\Id_{X}$.

\medskip

\textit{Representability of $\mathcal{Z}$.}
Now we are ready to show that $\mathcal{Z}$ is actually represented by an $S$-scheme.
Unfortunately we cannot show that $\Res_{T/S}(\mathcal{Y})$ is representable using \cite[Chapter 7.6, Theorem~4]{MR1045822},
because as $\mathcal{Y}$ does not need to be quasi-projective,
we cannot show that every finite set of points in $\mathcal{Y}$ is contained in an affine subset of $\mathcal{Y}$.
Therefore we show directly that $\mathcal{Z}$ is representable by using methods from the proof of \cite[Chapter 7.6, Theorem~4]{MR1045822}.

Note that if $U\subset \mathcal{Y}$ is open and $G$-invariant, then $\Res_{T/S}(U)^G$ is well defined, and moreover there is a natural map of functors
$\Res_{T/S}(U)^G\to\Res_{T/S}(\mathcal{Y})^G$, because $\Hom_T(W\times_S T,U)^G(W)$ is a subset of $\Hom_T(W\times_S T,\mathcal{Y})^G(W)$.
By \cite[Chapter 7.6, Proposition 2]{MR1045822} the morphism of functors $\Res_{T/S}(U)\to\Res_{T/S}(\mathcal{Y})$ is an open immersion.
As
\[
\Res_{T/S}(U)^G\cong \Res_{T/S}(\mathcal{Y})^G\times_{ \Res_{T/S}(\mathcal{Y})}\Res_{T/S}(U),
\]
and as open immersion are stable under base change,
$\Res_{T/S}(U)^G\to\Res_{T/S}(\mathcal{Y})^G$ is an open immersion.

Let $\cup U_i=\mathcal{Y}$ be a cover of $\mathcal{Y}$ by affine, $G$-invariant open subsets,
which exists, because the $G$-action on $\mathcal{Y}$ is good.
Consider any $U_i$.
As $T\to S$ is finite and flat, and $U_i$ is affine,
by \cite[Chapter 7.6, Theorem 4]{MR1045822},
$\Res_{T/S}(U_i)$ is represented by a scheme, and
by \cite[Proposition 3.1]{ MR1149171}, $(\Res_{T/S}(U_i))^G$ is represented by a subscheme of $\Res_{T/S}(U_i)$.
As ${\Res_{T/S}(\mathcal{Y}_L)^G\cong X}$ as functors,
$\Res_{T/S}(\mathcal{Y}_L)^G$ is representable by the scheme $X$.

If $U,V\subset \mathcal{Y}$ are two open, $G$-invariant subsets,
then $U\cap V\subset \mathcal{Y}$ is also an open, $G$-invariant subset,
and the open immersions $\Res_{T/S}(U\cap V)^G\to \Res_{T/S}(U)^G $ and $\Res_{T/S}(U\cap V)^G\to \Res_{T/S}(U)^G $
define a gluing data for $\Res_{T/S}(U)^G $ and $\Res_{T/S}(V)^G $.
On computes that in fact
\begin{equation} \label{xyz}
 \Res_{T/S}(U)^G \times_\mathcal{Z}\Res_{T/S}(V)^G=\Res_{T/S}(U\cap V)^G.
\end{equation}
Let $\mathcal{S}$ be the scheme constructed by gluing $X$ and the $\Res_{T/S}(U_i)^G$ as explained above.
We get a map $\iota:\mathcal{S}\to \mathcal{Z}$, because the gluing data is compatible with the open immersions $X\to \mathcal{Z}$ and $\Res_{T/S}(U_i)^G\to \mathcal{Z}$.
As $\iota\mid_X$ and the $\iota\mid_{\Res_{T/S}(U_i)^G}$ are open immersions, and equation (\ref{xyz}) holds pairwise for $X$ and the $\Res_{S/T}(U_i)^G$, $\iota$ is an open immersion.

We now want to show that $\iota$ is an equivalence of functors.
Considering the last paragraph of the proof of \cite[Chapter 7.6, Theorem 4]{MR1045822} it suffices to show the following.
For every field $F$ with a map $\Spec(F)\to S$ every $T$-morphism $f: \Spec(F)\times _S T\to \mathcal{Y}$ factors either through $X_L$ or through one of the $U_i$.
If $F$ lies over $K$, $f$ will factor through $X_L$. 
If $F$ lies over $k$, $f(\Spec(F)\times _S T)$ will only be a point topologically,
so $f$ will factor through every open neighborhood of that point.
As the $U_i$ cover $\mathcal{Y}$, there is a $U_i$ through which $f$ factors.

Altogether $\iota$ is an equivalence of functors, and therefore $\mathcal{Z}$ is represented by the $S$-scheme $\mathcal{S}$.

\medskip
Now we show that $\mathcal{Z}$ is a smooth $S$-scheme.
Note that it suffices to check smoothness locally.
By construction of the scheme representing $\mathcal{Z}$, 
every point in $\mathcal{Z}$ lies either in $X\cong \mathcal{Z}_K$, or in $\Res_{T/S}(U_i)^G$ for some $i$.
By assumption $X$ is smooth over $S$.
Moreover $\mathcal{Y}$ is smooth over $T$, i.e.~in particular the $U_i\subset \mathcal{Y}$ are smooth over $T$.
Hence by \cite[Chapter 7.6, Proposition 3]{MR1045822}, the $\Res_{T/S}(U_i)$ are smooth over $S$.
So by \cite[Proposition 3.4]{ MR1149171} the $\Res_{T/S}(U_i)^G$ are smooth over $S$.
Altogether $\mathcal{Z}$ is smooth over $S$.

We have seen that $X\cong \mathcal{Z}_K$, and $X$ is integral by assumption.
As $\mathcal{Z}$ is smooth over $S$, it is reduced, and flat over $S$, so there is no irreducible component only supported on the special fiber.
Altogether $\mathcal{Z}$ is integral.
This yields that $\mathcal{Z}\to S$ is a smooth model of $X$.

\medskip

\textit{Properties of $\Phi$.}
In order to show that $\Phi$ is separated,
take any valuation ring $R$ with quotient field $Q$, and any two morphisms $f_1,f_2\in \Hom(\Spec(R),\mathcal{Z})$ such that $f_1\mid _{\Spec(Q)}=f_2\mid _{\Spec(Q)}$,
and $\Phi \circ f_1=\Phi \circ f_2$.
Let $x\in \mathcal{X}$ be the image of the closed point of $R$.
As $R$ is a valuation ring, $\Phi \circ f_1=\Phi \circ f_2$ will factor through every open neighborhood of $x$, so we may assume
that it factors through $\Spec(A_i^G)\subset \mathcal{X}$ for some $G$-equivariant affine subset $U_i=\Spec(A_i)\subset \mathcal{Y}$.
Hence the $f_i$ factor through $\Phi^{-1}(\Spec(A_i^G))\cong \Res_{S/T}(U_i)^G$.
By \cite[Chapter 7.6, Proposition 5]{MR1045822} and \cite[Proposition 3.1]{ MR1149171}, $\Res_{S/T}(U_i)^G$ is separated over $S$,
hence $f_1=f_2$. So by the valuative criterion of separateness, $\Phi$ is separated.

\medskip

Now look at
$ \mathcal{Z}(S)=\Hom_T(T,\mathcal{Y})^G= \{ \sigma \in \Hom_T(T,\mathcal{Y})\mid g\circ \sigma \circ {g_T}^{-1}=\sigma\}$.
Let $\pi_T: T\to S$ be the quotient map induced by the Galois action on $T$.
Note that $\Phi(S): \mathcal{Z}(S)\to \mathcal{X}(S)$ maps $\sigma\in \mathcal{Z}(S)$ to
$\sigma_G\in \mathcal{X}(S)$ with $\sigma_G\circ \pi_T=\pi \circ \sigma$.

Take any $\sigma_G\in \mathcal{X}(S)$.
Let $\sigma_G': \mathcal{Y}\times_{\mathcal{X}}S\to \mathcal{X}$ be the pullback of $\sigma_G$,
$\pi':\pi^{-1}(S)\to S$ the pullback of $\pi$,
and
set $\varphi'\!:=\varphi \circ \sigma_G'$.
By the universal property of the fiber product, we have a one to one correspondence of sections $\sigma$ of $\varphi$ with ${\pi \circ \sigma =\sigma_G\circ \pi_T}$, and sections $\sigma'$ of $\varphi'$ with ${\pi ' \circ \sigma' = \pi_T}$.
Note that $\pi_T\circ \varphi'=\pi'$,
$\pi_T$ is separated,
and $\pi'$ is proper, because $\pi$ is a quotient map and hence proper.
So $\varphi'$ is proper by \cite[Proposition 12.58]{MR2675155}.
Hence without loss of generality we may assume that $\varphi$ is proper.

By assumption, $\mathcal{Y}_L\cong \mathcal{X}\times_S\Spec(L)$,
hence $\sigma_G$ induces a unique section $\sigma'$ of $\varphi\!\mid_{\mathcal{Y}_L}$ with $\pi \circ \sigma'=\sigma_G\circ \pi_T\!\mid_{\Spec(L)}$.
As $\varphi$ is proper, we get a unique section $\sigma$ of $\varphi$ with $\sigma\!\mid_{\Spec(L)}=\sigma'$.
As $S$ is separated, $\pi \circ \sigma = \sigma_G\circ \pi_T$.
We still need to show that $\sigma\in \mathcal{Z}(S)$, i.\thinspace e.~that $\sigma = g\circ \sigma \circ {g_T}^{-1}$.
Therefore one shows that $g\circ \sigma \circ {g_T}^{-1}$ is a section of $\varphi$ and $\pi \circ g\circ \sigma \circ {g_T}^{-1} = \sigma_G\circ \pi_T$,
and one concludes using the uniqueness of $\sigma$ with this properties.
Hence $\sigma$ is the unique element in $\mathcal{Z}(S)$ with $\Phi(S)(\sigma)=\sigma_G$,
i.e.~$\Phi (S)$ is bijective.

\medskip
\textit{Universal property.}
Now let $\mathcal{V}$ be a smooth, integral $S$-scheme and let ${\Psi:\mathcal{V}\to \mathcal{X}}$ be a dominant $S$-morphism.
Assume that there exists a $\Psi': \mathcal{V}\to \mathcal{Z}$ such that $\Phi \circ \Psi'=\Psi$.
As $\Psi$ is an $S$-morphism, it maps $\mathcal{V}_K$ to $X\cong \mathcal{X}_K$.
We have already seen that $\Phi\!\!\mid_{\mathcal{Z}_K}:\mathcal{Z}_K\to X$ is an isomorphism with inverse map ${\Phi\!\mid_{ \mathcal{Z}_K}}^{-1}$.
Therefore we have $\Psi'\!\!\mid_{\mathcal{V}_K}={\Phi\!\mid_{ \mathcal{Z}_K}}^{-1}\circ \Psi\!\!\mid_{\mathcal{V}_K}$.
As $\mathcal{V}_K$ is open and dense in $\mathcal{V}$, $\mathcal{V}$ is reduced, and $\mathcal{Z}$ is separated over $\mathcal{X}$, $\Psi'$ is unique on $\mathcal{V}$  by \cite[Corollary~9.9]{ MR2675155}.

Now we construct $\Psi'$. 
First we need to show some facts concerning $\mathcal{Y}$ and the $G$-action on $\mathcal{Y}$.
Consider $\mathcal{X}_T:=\mathcal{X}\times_S T$ and the following diagram.
\[
  \xymatrix{
\mathcal{Y} \ar@{-->}[rd]^{h}  \ar@/_5mm/[ddr]_{\varphi} \ar@/^5mm/[drr]^\pi \\
& \mathcal{X}_T \ar[r]^{p_{\mathcal{X}}} \ar[d]_{p_T} & \mathcal{X} \ar[d] \\
& T\ar[r]_{\pi_T} \ar@{}[ur]|{\square} & S
}
\]
Here $p_{\mathcal{X}}$ and $p_T$ are the projection maps.
Note that the diagram commutes, so there is a unique $h$  with $p_T\circ h= \varphi$ and $p_\mathcal{X}\circ h=\pi$.
As $p_\mathcal{X}$ and $\pi$ are finite, by \cite[Proposition 12.11]{ MR2675155} $h$ is finite, too.
As $\mathcal{X}$ is flat over $S$, the $G$-action on T induces a $G$-action on $\mathcal{X}_T$ such that $p_T$ is $G$-equivariant and $\mathcal{X}_T/G=\mathcal{X}$,
see \cite[Expos\'e V, Proposition 1.9]{MR0217087}.
As $\varphi$ and $p_T$ are $G$-equivariant, $h$ is $G$-equivariant, too.

Let $n: \mathcal{X}_T^n\to \mathcal{X}_T$ be the normalization.
By assumption, $\mathcal{Y}$ is integral and smooth over $T$, so in particular normal.
As $\mathcal{Y}_L=X_L=\mathcal{X}_L$, $h$ is generically an isomorphism, and therefore dominant.
So the universal property of normalization induces a unique morphism $s: \mathcal{Y}\to \mathcal{X}_T^n$ such that $n\circ s =h$.
Note that $s$ is finite, because $h$ and $n$ are finite,
and an isomorphism on $X_L\subset \mathcal{Y}$.
Altogether $s$ is a finite birational morphism between integral normal schemes. That means, by 
\cite[Corollary 12.88]{ MR2675155} it is an isomorphism.
So we may assume that $h=n$ and $\mathcal{Y}= \mathcal{X}_T^n$.

Back to $\mathcal{V}$ and $\Psi$. Consider the following cartesian diagram.
\[
  \xymatrix{
&\mathcal{V}_T \ar@{-->}_{\Psi_T}[dl]\ar[r]^{\pi_\mathcal{V}}  \ar[d]^{p} & \mathcal{V} \ar[d]^{\Psi} \\
 \mathcal{Y}\ar[r]_n&\mathcal{X}_T \ar[r]  \ar@{}[ur]|{\square} & \mathcal{X}
}
\]
with $\mathcal{V}_T\!:=\mathcal{V}\times _S T=\mathcal{V}\times_{\mathcal{X}} \mathcal{X}_T$, $\pi_{\mathcal{V}}$ and $p$ the projection maps.
As $\mathcal{V}$ is smooth over $S$, so in particular flat,
the $G$-action on T induces a $G$-action on $\mathcal{V}_T$ such that $\mathcal{V}_T\to T$ is $G$-equivariant
and $\mathcal{V}_T/G=\mathcal{V}$,
see \cite[Expos\'e~V, Proposition~1.9]{MR0217087}.
% This action is generated by $\Id_\mathcal{V}\times g_T\in \Aut(\mathcal{V}_T)$.
By construction $p$ is $G$-equivariant.

It might happen that $\mathcal{V}_T$ is not connected.
Let $\mathcal{V}_T=U_1\sqcup \dots \sqcup U_m$, with $U_i\subset \mathcal{V}_T$ the connected components.
As $\mathcal{V}=\mathcal{V}_T/G$ is connected, $G$ acts transitively on the connected components.
As $\Psi$ is dominant, the same holds for $p$.
Note that $\mathcal{X}_T$ is connected, because it is flat over $T$ and generically isomorphic to the $L$-variety $X_L$.
Hence there exists at least one component $U_i$ such that $p\!\!\mid_{U_i}$ is dominant.
As $G$ acts transitively on $\mathcal{V}_T$ and $p$ is $G$-equivariant, $p\!\!\mid _{U_j}$ is dominant for every component $U_j$.
By assumption $\mathcal{V}$ is smooth over $S$, so $\mathcal{V}_T$ is smooth over $T$.
Hence every component $U_i$ of $\mathcal{V}_T$ is normal.
So by the universal property of normalization there are unique morphisms $\Psi_T\!\mid_{U_i}: U_i\to \mathcal{Y}$ such that $n\circ \Psi_T\!\mid_{U_i}=p\!\mid_{U_i}$.
This defines a unique morphism $\Psi_T$ on all of $\mathcal{V}_T$ such that $n\circ \Psi_T=p$.
As $p$ and $n$ are $G$-equivariant,
$\Psi_T$ is $G$-equivariant, too.

Take any $W\in (\Sch/S)$, $f\in \mathcal{V}(W)$.
By the universal property of the fiber product, there is a unique $\tilde{f}\in \Hom_T(W\times_ST,\mathcal{V}_T)$ such that $f\circ p_W=\pi_{\mathcal{V}}\circ \tilde{f}$.

One checks easily that $\tilde{f}$ is $G$-equivariant.
Set $\Psi'(f):=\Psi_T\circ \tilde{f}$.
As $\Psi_T$ is $G$-equivariant,  $\Psi'(f)\in \mathcal{Z}(W)$.
It is easy to check that this defines a map of functors, so we obtain an $S$-morphism $\Psi\in \Hom_S(\mathcal{V}, \mathcal{Z})$.

We still need to check that $\Psi=\Phi \circ \Psi'$.
Therefore it suffices to check that for all $f\in \mathcal{V}(W)$, $W\in (\Sch/S)$ flat,
$\Psi(f)=\Phi \circ \Psi'(f)$ .
Note that the following diagram commutes.
\[
  \xymatrix{
  & & & \mathcal{V}\ar[ddd]^\Psi\\
 & \mathcal{V}_T \ar[r]^{\Psi_T} \ar@/^0,5pc/[rru]^{\pi_\mathcal{V}}& \mathcal{Y} \ar[d] \ar[ddr]^\pi  \\ 
W \ar@/^2,5pc/[rrruu]^f\ar[rrd] \ar@/_2pc/[rrrd]_{\Phi(\Psi'(f))}& {W\times_ST} \ar[u]^{\tilde{f}}\ar[ur]_{\!\!\Psi'(f)}\ar[r]\ar[l]_{p_W} & T \ar[d] \\
& & S & {\mathcal{X}} \ar[l]
}
\]
One observes that $\Phi(\Psi'(f))=\Psi \circ f=\Psi(f)$, which we wanted to show.

\medskip

We still need to check that $\mathcal{Z}$ is unique up to a unique isomorphism with its properties. Assume there is a $\mathcal{Z}'$ and a morphism $\Phi':\mathcal{Z}'\to \mathcal{X}$ having the same properties as $\mathcal{Z}$ and $\Phi$.
So we get unique morphisms $\alpha: \mathcal{Z}\to \mathcal{Z}'$ and $\alpha': \mathcal{Z}'\to \mathcal{Z}$ with $\Phi \circ \alpha'=\Phi'$ and $\Phi'\circ \alpha= \Phi$.
Note that $\Phi \circ (\alpha'\circ \alpha)= \Phi' \circ \alpha= \Phi$. But $\Id_\mathcal{Z}$ is unique with $\Phi\circ \Id_\mathcal{Z}=\Phi$, so $\alpha'\circ \alpha=\Id_\mathcal{Z}$. Similarly one gets $\alpha\circ \alpha'=\Id_{\mathcal{Z}'}$.
So $\alpha$ is the unique isomorphism over $\mathcal{X}$ of $\mathcal{Z}$ and $\mathcal{Z}'$.

\medskip

\textit{The case that $\mathcal{Y}$ is a weak N\'eron model.}
Assume that $\varphi: \mathcal{Y}\to T$ is a weak N\'eron model of $X_L$.
Hence $\varphi$ is separated, so by \cite[Expos\'e~V, Proposition~1.5]{MR0217087} $\mathcal{X}$ is separated over $S$.
As $\Phi$ is separated, $\mathcal{Z}\to S$ is separated, too.
Hence to show that $\mathcal{Z}\to S$ is a weak N\'eron model of $X$, we still need to show that
\begin{align*}
 \mathcal{Z}(S)=\Hom_T(T,\mathcal{Y})^G &\to \mathcal{Z}_K(K)=\Hom_L(\Spec(L),X_L)^G\cong X(K)
\end{align*}
is a bijection.
This map is injective, because $\mathcal{Y}$ is a separated $T$-scheme. Take any $\sigma'\in \mathcal{Z}_K(K)$.
As $\mathcal{Y}$ is a weak N\'eron model of $X_L$, 
$ \mathcal{Y}(T)\cong X_L(L)$,
so there is a $\sigma \in \Hom_T(T,\mathcal{Y}) $ with $\sigma \!\mid_{\Spec(L)}=\sigma'$.
As $g_T^{-1}$ maps $\Spec(L)$ to itself, we get 
$g \circ \sigma \circ g_T^{-1} \!\mid_{\Spec(L)}
\sigma \!\mid_{\Spec(L)}$.
As $\mathcal{Y}$ is a separated $T$-scheme, $g\circ \sigma \circ g_T^{-1}=\sigma$, i.\thinspace e.~$\sigma\in \mathcal{Z}(S)$.
Hence the map $\mathcal{Z}(S)\to \mathcal{Z}_K(K)$ is surjective.

\end{proof}

In \cite[Theorem 4.2]{ MR1149171} the following statement is proven:
\begin{thm*}
Let $L/K$ be a tame Galois extension, $\mathcal{O}_L$ the ring of integers of $L$, and ${T:=\Spec(\mathcal{O}_L)}$.
Let $X$ be an abelian variety over $K$.
Then there is a good ${\Gal(L/K)}$-action on the N\'eron model $\varphi: \mathcal{Y}\to T$ of $X_L$ extending the Galois action on $X_L$,
and $\mathcal{Z}\to S$ given by Construction \ref{construction Z} is the N\'eron model of  $X$.
\end{thm*}

Note that the N\'eron model of an Abelian variety is uniquely determined by a universal property.
In \cite{ MR1149171} this universal property is used to show that $\mathcal{Z}$ is the N\'eron model of $X$.
As we do not have a universal property for weak N\'eron models in general, we had to use different methods to prove the universal property in Theorem~\ref{main theorem}.

Moreover, N\'eron models of Abelian Varieties are quasi-projective.
In Theorem~\ref{main theorem} we do not assume that $\mathcal{Y}$ is quasi-projective, which makes the proof of the representability of $\mathcal{Z}$ less straightforward.

\begin{rem}
 If we do not assume that $\varphi$ is smooth in Theorem \ref{main theorem}, we can modify Construction \ref{construction Z} by considering $\Sm(\mathcal{Y}/S)$, the smooth locus of $\mathcal{Y}$ over $S$, instead of $\mathcal{Y}$.
 This is well defined, because the $G$-action restricts to $\Sm(\mathcal{Y}/S)$.
We will get a smooth model of $X$ with an $S$-morphism $\Phi$ as in Theorem \ref{main theorem}.
 Note that the map $\Phi(S): \mathcal{Z}(S)\to \mathcal{X}(S)$ will be injective in this case, but in general not surjective.
 
 Nevertheless, if we assume that the smooth locus of $\mathcal{Y}$ over $S$ is a weak N\'eron model of $X_L$, the modified $\mathcal{Z}$ will be a weak N\'eron model of $X$.
 This is in particular the case if $\mathcal{Y}$ is regular and $\varphi$ is proper.
\end{rem}

\begin{rem}
If we do not assume in Theorem \ref{main theorem} that the Galois extension $L/K$ is tame, 
we cannot show that $\mathcal{Z}$ is smooth, because then \cite[Proposition 3.4]{ MR1149171} does not hold,
see \cite[Example 4.3]{ MR1149171}.
\end{rem}

\section{Local Studies}
\label{local studies}
\subsection{Main Lemma}

\begin{lem}\label{main lemma} 
Assumption and notation as in Theorem \ref{main theorem}.
Let $\mathcal{Y}^G$ be the fixed locus of the $G$-action on $\mathcal{Y}$.
Then there is a $k$-morphism
\[
b: \mathcal{Z}_k\to \mathcal{Y}^G
\]
such that for every point $y\in \mathcal{Y}^G$ with residue field $\kappa(y)$ 
the inverse image of $y$ is isomorphic to $\mathbb{A}^{m}_{\kappa(y)}$ as $\kappa(y)$-schemes for some $m\in \mathbb{N}$.
\end{lem}

\begin{proof}
As $L/K$ is a tame Galois extension, $G=\mathbb{Z}/r\mathbb{Z}$ with $r$ prime to $\Char(k)$.
Let the $G$-action on $\mathcal{Y}$ be given by $g\in \Aut(\mathcal{Y})$, and that on $T$ by $g_T\in \Aut(T)$.
By \cite[Proposition 3.1]{ MR1149171}, $\mathcal{Y}^G$ is a closed subscheme of $\mathcal{Y}$.
Take any $S$-scheme $W$, $f\in \mathcal{Y}^G(W)$. Then $g_T\circ \varphi \circ f= \varphi \circ g \circ f= \varphi \circ f$,
i.e.~$\varphi \circ f\in T^G(W)$ As $T^G=\Spec(k)$, 
$\mathcal{Y}^G$ is a closed subscheme of $\mathcal{Y}_k\subset \mathcal{Y}$, so in particular a $k$-scheme of finite type.

\medskip

To construct $b$, let $W\in (\Sch/k)$ be any $k$-scheme, $w:W\to \Spec(k)$ the structure map.
Recall the construction of $\mathcal{Z}$ in Construction \ref{construction Z}.
Set
\begin{align*}
b(W):\ \mathcal{Z}_k(W)=\Hom_T(W\times_ST,\mathcal{Y})^G&\to \mathcal{Y}^G(W);
 f\mapsto f\circ i_W
\end{align*}
with $i_W: W\hookrightarrow W\times_S T$ the inclusion of the special fiber.
By construction, $b(W)(f)\in \Hom_T(W,\mathcal{Y})$.
We have
$ g\circ f\circ i_W = f\circ (\Id_W\times g_T) \circ i_W=f\circ i_W$.
Here the first equation holds, because $f$ is $G$-equivariant, and
the second, because the action on the special fiber $i_W(W)\subset W\times_ST$ is trivial.
Hence $b(W)(f)\in \mathcal{Y}^G(W)$.
It is obvious that $b$ is functorial, so we get the required $k$-morphism.

\medskip

Let $y\in \mathcal{Y}^G$ be any point with residue field $\kappa(y)$,
$j_y: \Spec(\kappa(y))\hookrightarrow \mathcal{Y}^G\subset \mathcal{Y}$ be the immersion of the point $y$.
Note that $b^{-1}(y)$ is defined by the following cartesian diagram.
\[
  \xymatrix{
 b^{-1}(y) \ar[r] \ar@{}[dr]|\square \ar[d] & \mathcal{Z}_k \ar[d]^b \\
\Spec(\kappa (y))\ar[r]_{\ \ \ j_y}& \mathcal{Y}^G
}
\]
Take any affine $\kappa(y)$-scheme $W=\Spec(A)\in (\Sch/\kappa(y))$ with structure map ${\omega:W\to \Spec(\kappa(y))}$.
By the universal property of the fiber product we obtain
\begin{align*}
 b^{-1}(y)(W)&=\{f\in \mathcal{Z}_k(W)\mid b\circ f=j_y \circ \omega\}\\
&=\{f\in \Hom_T(W\times_ST,\mathcal{Y})^G\mid f\circ i_W=j_y \circ \omega\}.
\end{align*}
As $\mathcal{Y}^G$ is a subscheme of $\mathcal{Y}_k$, $W$ is a $k$-scheme with structure map ${\varphi \circ j_y\circ \omega}$.
Recall that $G$ acts on $\Hom_T(W\times_ST,\mathcal{Y})$ by sending $f\in \Hom_T(W\times_ST,\mathcal{Y})$ to $g\circ f\circ (\Id_W\times g_T)^{-1}$.
Set $R:=\mathcal{O}_L$. Hence $R^G=\mathcal{O}_K$.
We have
\begin{align*}
W\times_ST&=W \times _{\Spec(k)} \Spec(k)\times_{S} T
=W\times _{\Spec(k)} \Spec(k\otimes_{R^G}\!R)\\
&\cong\Spec(A\otimes_{k}k[t]/(t^r))
=\Spec(A[t]/(t^r)).
\end{align*}
To compute $k\otimes_{R^G}R$ we use Lemma \ref{ap2}.
This lemma also implies that
\[
 \alpha\!:=(\Id\times g_T)^\#:A[t]/(t^r)\to A[t]/(t^r);\ p(t)\mapsto p(\mu t)
\]
for a primitive $r$-th root of unity $\mu \in k\subset \kappa(y)$.
Note that 
\[
r_W:=i_W^\#:A[t]/(t^r)\to A;\ p(t)\mapsto p(0).
\]
One observes that $f$ sends all points in $\Spec(A[t]/(t^r))$ to $y\in \mathcal{Y}$, so it factors uniquely through $\Spec(\mathcal{O}_{\mathcal{Y},y})$,
i.\thinspace e.~there is a unique morphism 
$\tilde{f}$
such that the following diagram commutes.
\[
  \xymatrix{
&  \Spec({\mathcal{O}}_{\mathcal{Y},y})\ar[d]^{j}\\
\Spec(A[t]/(t^r)) \ar[r]_{\ \ \ \ \ \ \  f} \ar[ur]^{\tilde{f}} & \mathcal{Y}
}
\]
Let $r_y\!:=i_y^\#: {\mathcal{O}}_{\mathcal{Y},y} \to \kappa(y)$ be the residue map.
Note that $j\circ i_y=j_y$.
As $f\circ i_W = j_y \circ \omega$,
we get that $ j\circ \tilde{f} \circ i_W = j \circ i_y \circ \omega$.
The fact that $j$ is a monomorphism implies that $\tilde{f}\circ i_W = i_y \circ \omega$.

As $y$ lies in $\mathcal{Y}^G$, there is an induced $G$-action on $\Spec({\mathcal{O}}_{\mathcal{Y},y})$ given by some map $\tilde{g}\in \Aut(\Spec({\mathcal{O}}_{\mathcal{Y},y}))$ with $\tilde{g}^r=\Id$ and $\alpha_y\in \Aut({\mathcal{O}}_{\mathcal{Y},y})$ with $\alpha_y^r=\Id$, respectively,
such that $j$ is $G$-equivariant.
Hence $f=g\circ f \circ (\Id_W\times g_T)^{-1}$ implies that
\[
j\circ (\tilde{g}\circ \tilde{f} \circ (\Id_W \times g_T)^{-1})=g\circ j\circ \tilde{f} \circ (\Id_W \times g_T)^{-1}=f.
\]
As $\tilde{f}$ is unique with this property, $\tilde{g}\circ \tilde{f} \circ (\Id_W \times g_T)^{-1}=\tilde{f}$.\\
Assume that $\tilde{f}\in \Hom_T(\Spec(A[t]/(t^r)),\Spec({\mathcal{O}}_{\mathcal{Y},y}))$ with 
$\tilde{f}\circ i_W=i_y\circ \omega$, and $\tilde{g}\circ \tilde{f} \circ (\Id_W\times g_T)^{-1}=\tilde{f}$.
Then $f\!:=j\circ \tilde{f}\in \Hom_T(\Spec(A[t]/(t^r)),\mathcal{Y})$, and we have that
$g \circ f \circ (\Id_W \times g_T)^{-1}= f$
as well as
$ f\circ i_W=j_y \circ \omega$.
Altogether, we obtain
\begin{align*}
 b^{-1}(y)(W) =\{ & \tilde{f}\in  \Hom_T( \Spec(A[t]/(t^r)),\Spec({\mathcal{O}}_{\mathcal{Y},y})) \\
& \mid \tilde{f}\circ i_W=i_y\circ \omega \text{ and } \tilde{g}\circ \tilde{f} \circ (\Id_W\times g_T)^{-1}=\tilde{f}\}\\
=\{ & a\in \Hom_{R}({\mathcal{O}}_{\mathcal{Y},y}, A [t]/(t^r))\\
& \mid r_W\circ a = \omega^\# \circ r_y \text{ and } \alpha^{-1} \circ a \circ \alpha_y = a\}.
\end{align*}

\medskip

Now consider $\mathcal{O}_{\mathcal{Y},y}^G$. Note that,
as $\mathcal{O}_{\mathcal{Y},y}$ is an $R$-module, it contains lifts of all roots of unity and hence by Remark~\ref{rem not henselian}
we get that Remark \ref{rem ap2} and Lemma \ref{ap2} hold.
Let $i^G: {\mathcal{O}}_{\mathcal{Y},y}^G\hookrightarrow {\mathcal{O}}_{\mathcal{Y},y}$ be the inclusion,
and $r_y^G: {\mathcal{O}}_{\mathcal{Y},y}^G\to \kappa(y)$ the residue map.
We have $r_y^G=r_y\circ i^G$.
Consider the following diagram.
\begin{align}\label{keydiagram}
  \xymatrix{
&  & A[t]/(t^r)\\
{\mathcal{O}}_{\mathcal{Y},y}\ar@/^5mm/[urr]^a \ar[r]^/-4mm/{\rho_1} & {\mathcal{O}}_{\mathcal{Y},y}\otimes_{{\mathcal{O}}_{\mathcal{Y},y}^G}\!\!\! \kappa(y) \ar@{-->}[ru]^{\tilde{a}}\\
{\mathcal{O}}_{\mathcal{Y},y}^G\ar@{^(->}[u]^{i^G} \ar[r]_{r_y^G} & \kappa(y) \ar[u]_/-1mm/{\rho_2} \ar@/_5mm/[uur]_{i_0\circ \omega^\#}
}
\end{align}
Here $a\in b^{-1}(y)(W)$ as described before,
$\rho_1$ and $\rho_2$ are the morphisms we get from the definition of tensor product, and
$ i_0: A\to A[t]/(t^r); \ c\mapsto c$.
Note that $r_W\circ i_0=\Id$, and $i_0\circ r_W\!\mid_{i_0(A)}=\Id$.
One observes that for every $u\in {\mathcal{O}}_{\mathcal{Y},y}^G$, $(\alpha^{-1}\circ a) (u)=(\alpha^{-1}\circ a \circ \alpha_y)(u)=a(u)$.
Set $a(u)=\sum_{i=0}^{r-1}a_i t^i$ for some $a_i\in A$.
Hence $(\alpha^{-1}\circ a)(u)=\sum_{i=0}^{r-1}{\mu}^{-i}a_it^i$.
Comparing coefficients yields $a(u)=a_0$, i.\thinspace e.~$ a({\mathcal{O}}_{\mathcal{Y},y}^G)\subset i_0(A)$.
Using in addition that $r_y^G=r_y\circ i^G$ and $r_W\circ a=\omega^\# \circ r_y$, we obtain
\begin{align*}
 i_0\circ \omega^\# \circ r_y^G= i_0 \circ \omega^\#\circ r_y \circ i^G
 = i_0 \circ r_W \circ a \circ i^G
= i_0 \circ r_W\!\mid_{i_0(A)}\circ a \circ i^G
= a\circ i^G.
\end{align*}
Hence by the universal property of tensor product there is a unique $\tilde{a}$ such that diagram (\ref{keydiagram}) commutes.

Now, $G$ acts on $ {\mathcal{O}}_{\mathcal{Y},y}\otimes_{{\mathcal{O}}_{\mathcal{Y},y}^G}\!\!\! \kappa(y)$
given by
${\tilde{\alpha}_y\in\Aut( {\mathcal{O}}_{\mathcal{Y},y}\otimes_{{\mathcal{O}}_{\mathcal{Y},y}^G} \!\!\!\kappa(y))}$,
such that $\rho_1$ and $\rho_2$ are $G$-equivariant, see Lemma \ref{ap2}.
As $\alpha^{-1}\circ a\circ \alpha_y=a$, we get
\[
 (\alpha^{-1}\circ \tilde{a}\circ \tilde{\alpha}_y) \circ \rho_1 = \alpha^{-1}\circ \tilde{a}\circ \rho_1\circ {\alpha}_y = a 
\]
and, using that $G$ acts trivially on $i_0(A)$, we obtain
\begin{align*}
( \alpha^{-1}\circ \tilde{a}\circ \tilde{\alpha}_y) \circ \rho_2 &= \alpha^{-1} \circ \tilde{a}\circ \rho_2
= \alpha^{-1}\circ i_0 \circ \omega^\#=i_0 \circ \omega^\#.
\end{align*}
As $\tilde{a}$ is unique with these properties,
$\alpha^{-1}\circ \tilde{a}\circ \tilde{\alpha}_y=\tilde{a}$.

Denote by $\tilde{r}: k\otimes_{R^G}\!R\cong k[t]/(t^r) \to A \otimes_k k\otimes_{R^G}\! R\cong A[t]/(t^r)$ the canonical map given by the properties of the tensor product.
We have $\tilde{r}(t)=t$.
The $R$-structure of $A[t]/(t^r)$ is given by $\tilde{r}\circ \rho_R$, with $\rho_R: R\to k\otimes _{R^G}R$ the canonical map.
The $R$-structure of ${\mathcal{O}}_{\mathcal{Y},y}$ is given by $\beta_y:=(\varphi\circ j)^\#$.
As $a$ is an $R$-morphism, we obtain the following commutative diagram.
\begin{align}
\label{blubb}
  \xymatrix{
& k[t]/(t^r) \ar@{-->}[d]^{\tilde{\beta_y}} \ar[ld]_{\tilde{r}} & R\ar[d]^{\beta_y} \ar[l]_/-4mm/{\rho_R}\\
A[t]/(t^r) &  {\mathcal{O}}_{\mathcal{Y},y}\otimes_{{\mathcal{O}}_{\mathcal{Y},y}^G}\!\!\! \kappa(y) \ar[l]_/2mm/{\tilde{a}} & {\mathcal{O}}_{\mathcal{Y},y} \ar[l]_/-6mm/{\!\!\!\rho_1} \ar@/^5mm/[ll]^{a}
}
\end{align}
By Remark \ref{rem ap2}, $R^G\subset R$ is a local subring having the same residue field as $R$.
As $\beta_y$ is $G$-equivariant, it maps $R^G$ to ${{\mathcal{O}}_{\mathcal{Y},y}^G}$.

It is easy to check that the following diagram commutes.
 \[
  \xymatrix{
\kappa(y)\otimes_{{\mathcal{O}}_{\mathcal{Y},y}^G}\!\!\!{\mathcal{O}}_{\mathcal{Y},y}& {\mathcal{O}}_{\mathcal{Y},y}\ar[l]_{\ \ \ \ \ \ \ \rho_1}\\
\kappa(y)  \ar[u]^{\rho_2}& k\otimes_{R^G}\!R \ar@{-->}[ul]_/-5mm/{\tilde{\beta_y}}& R \ar[l]_/-3mm/{\rho_R} \ar[lu]_/-4mm/{\beta_y}\\
\ar@{}[drrr]_{} & k \ar@{_(->}[ul]\ar[u] & R^G\ar@{_(->}[u] \ar[l] \ar[rd]^/-2mm/{\!\!\!\beta_y\mid_{R^G}} \ar@{}[ul]|{\square}\\
 & & & {\mathcal{O}}_{\mathcal{Y},y}^G \ar@/_13mm/[uuull]_{i^G}\ar@/^12mm/[llluu]^{r_y^G}
}
\]
Hence the universal property of tensor product induces a unique $k$-morphism $\tilde{\beta_y}$ with 
$\tilde{\beta}_y\circ \rho_R =\rho_1 \circ \beta_y$.
Looking at diagram (\ref{blubb}) again, we get
\[
 \tilde{r} \circ \rho_R= a \circ \beta_y =\tilde{a} \circ \rho_1 \circ \beta_y = \tilde{a} \circ \tilde{\beta_y} \circ \rho_R.
\]
As $\rho_R$ is surjective, $\tilde{r}=\tilde{a} \circ \tilde{\beta_y}$, i.\thinspace e.~$\tilde{a}$ preserves the $k[t]/(t^r)$-structure
given by $\tilde{\beta}_y$
on ${{\mathcal{O}}_{\mathcal{Y},y}\otimes_{{\mathcal{O}}_{\mathcal{Y},y}^G}\!\!\! \kappa(y)}$,
and on $A[t]/(t^r)$ given by $\tilde{r}$.\\
Using the universal property of the tensor product, and that ${r_y\circ i^G=r_y^G}$, we get a unique morphism $\tilde{r}_y: {\mathcal{O}}_{\mathcal{Y},y}\otimes_{{\mathcal{O}}_{\mathcal{Y},y}^G}\!\!\! \kappa(y)\to \kappa(y)$,
such that ${\tilde{r}_y\circ \rho_1 =r_y}$ and $\tilde{r}_y\circ \rho_2=\Id$.
Using that $r_W\circ a=\omega^\# \circ r_y$, we get
\begin{align*}
 &(r_W \circ \tilde{a}) \circ \rho_1 = r_W\circ a \text{ and }(\omega^\# \circ \tilde{r}_y) \circ \rho_1 =\omega^\# \circ {r}_y = r_W\circ a, \\
&(r_W \circ \tilde{a}) \circ \rho_2 =r_W \circ i_0 \circ \omega^\# =\omega^\# \text{ and }(\omega^\# \circ \tilde{r}_y) \circ \rho_2=\omega^\#.
\end{align*}
Moreover,
$ r_W\circ a\circ i^G= r_W\circ i_0 \circ \omega^\# \circ r_y^G =\omega^\# \circ r_y^G$,
hence by the universal property of the tensor product there is a unique morphism $v:{\mathcal{O}}_{\mathcal{Y},y}\otimes_{{\mathcal{O}}_{\mathcal{Y},y}^G}\!\!\! \kappa(y) \to A$ such that
$v\circ \rho_1= r_W \circ a$ and $v\circ \rho_2= \omega^\#$, in particular
$ r_W\circ \tilde{a} =v =\omega^\# \circ \tilde{r}_y$,
meaning that $\tilde{a}$ is a $\kappa(y)$-morphism.

Note that for a given morphism 
$ \tilde{a}\in \Hom _{k[t]/(t^r)}({\mathcal{O}}_{\mathcal{Y},y}\otimes_{{\mathcal{O}}_{\mathcal{Y},y}^G} \!\!\!\kappa(y),A[t]/(t^r))$,
we have that $ a\!:=\tilde{a} \circ \rho_1 \in \Hom_R({\mathcal{O}}_{\mathcal{Y},y},A[t]/(t^r))$.
If $\alpha^{-1}\circ \tilde{a} \circ \tilde{\alpha}_y=\tilde{a}$, then
$ \alpha^{-1}\circ a \circ \alpha_y =a$.
If we assume furthermore that $r_W\circ \tilde{a} =\omega^\# \circ \tilde{r}_y$, then
$ r_W \circ a =  \omega^\# \circ r_y$.
So altogether
\begin{align*}
b^{-1}(y)&(W)=\{  \tilde{a}\in \Hom_{k[t]/(t^r)}({\mathcal{O}}_{\mathcal{Y},y}\otimes_{{\mathcal{O}}_{\mathcal{Y},y}^G} \kappa(y), A [t]/(t^r))\\
& \mid \tilde{a}\circ \rho_2=i_0 \circ \omega^\# \text{ and } r_W\circ \tilde{a} = \omega^\# \circ \tilde{r}_y \text{ and } \alpha^{-1} \circ \tilde{a} \circ \tilde{\alpha}_y = \tilde{a}\}.
\end{align*}
Note that $\tilde{a}\circ \rho_2=i_0 \circ \omega^\#$ is actually redundant.

\medskip

By Lemma \ref{ap2},
${\mathcal{O}}_{\mathcal{Y},y}\otimes_{{\mathcal{O}}_{\mathcal{Y},y}^G}\!\!\!\kappa(y)\cong \kappa(y)[ x_0,\dots, x_m]/\mathfrak{I}$,
\begin{align*}
\tilde{\alpha}_y(p(x_0,\dots,x_m))=p(\mu^{\ell_0}x_0,\dots,\mu^{\ell_{m}}x_{m})
\end{align*}
for $p(x_0,\dots,x_m) \in \kappa(y)[ x_0,\dots, x_m]/\mathfrak{I} $, $\mu\in k\subset \kappa(y)$ a primitive $r$-th root of unity,
$\ell_i\in \{1,\dots,r-1\}$, $m \in \mathbb{N}$, and
$\mathfrak{I}\subset \kappa(y)[x_0,x_1,\dots ,x_m]$ is the ideal generated by monomials of the form $x_0^{s_0}\dots x_m^{s_m}$ such that $\ell_0s_0+\dots + \ell_ms_m=rs$, $s\in \mathbb{N}$.

We now want to show that we can assume that $x_0=\tilde{\beta}_y(t)$.
By Lemma~\ref{ap1} there is a local parameter $t'\in R$ with $g_T^\#(t')=\mu t'$, $\mu \in R$ a primitive $r$-th root of unity.
Looking at the proof of Lemma \ref{ap2} we may assume that $\rho_R(t')=t$, i.e. ${\tilde{\beta_y}(t)=\tilde{\beta_y}\circ \rho_R(t')=\rho_1\circ \beta_y(t')}$.
So looking again at the proof of Lemma~\ref{ap2} it suffices to find a regular system of parameters $y_0,\dots,y_{m'}\in{\mathcal{O}}_{\mathcal{Y},y}$
such that ${\alpha_y (y_i)=\mu^{\ell_i}y_i}$, $\ell_i\in \{0,\dots, r-1\}$, $\ell_0\neq 0$, and $y_0={\beta}_y(t')=:\tilde{t}$.
Note that
$ \alpha_y(\tilde{t})=\beta_y(g_T^\#(t'))=\mu \tilde{t}$.
Hence using Lemma \ref{ap1} it suffices to show that $\tilde{t}\in \Em_y$, and $\tilde{t}\neq 0\mod \Em_y^2$ for the maximal ideal $\Em_y\subset {\mathcal{O}}_{\mathcal{Y},y}$.
As $y$ lies in $\mathcal{Y}^G\subset \mathcal{Y}_k$, $\tilde{t}\in \Em_y$.
Let $U=\Spec(C) \subset \mathcal{Y}$ be an affine neighborhood of $y$, $\mathfrak{p}\subset C$ the defining prime ideal of $y$. Choose a maximal ideal $\Em\subset C$ with $\mathfrak{p}\subset \Em$, let $y'$ be the corresponding closed point.
By \cite[Proposition 17.5.3]{ MR0238860}, ${\hat{\mathcal{O}}_{\mathcal{Y},y'}\cong R \Pol \tilde{y}_1,\dots, \tilde{y}_n\Por}$ as $R$-module.
So $\tilde{t}\neq 0\in {\hat{\mathcal{O}}}_{\mathcal{Y},y'}/ \!\Em^2 \cong {\mathcal{O}}_{\mathcal{Y},y'}/ \!\Em^2$.
As $\mathfrak{p}\subset \!\Em$, $\tilde{t}\neq 0\in {\mathcal{O}}_{\mathcal{Y},y'}/ {\mathcal{O}}_{\mathcal{Y},y'}\mathfrak{p}^2$.
As ${\mathcal{O}}_{\mathcal{Y},y'}/ {\mathcal{O}}_{\mathcal{Y},y'}\mathfrak{p}^2\subset {\mathcal{O}}_{\mathcal{Y},y}/\!\Em_y^2$ as $R$-modules, we have $\tilde{t}\neq 0 \mod \Em_y^2$.
This is what we wanted to show, so we may assume that $\tilde{\beta_y}$ is the $k$-morphism sending $t\in k[t]/(t^r)$ to $x_0$.

\medskip

Now chose any $\tilde{a}\in b^{-1}(y)(W)$. For $j\in \{1,\dots, m\}$ we have
\[
 \tilde{a}(x_j)=\sum_{i=0}^{r-1}a_{ij}t^i\in A[t]/(t^r)
\]
for some $a_{ij}\in A$.
Using $r_W \circ \tilde{a}=\omega^\# \circ \tilde{r}_y$, we obtain
\[
 a_{0j}=r_W(\tilde{a}(x_j))=\omega^\#(\tilde{r}_y (x_j))=\omega^\#(0)=0.
\]
From $\alpha^{-1} \circ \tilde{a} \circ \tilde{\alpha}=\tilde{a}$ we get
\[
\sum_{i=1}^{r-1} \mu^{\ell_j-i}a_{ij}t^i=(\alpha^{-1} \circ \tilde{a} \circ \tilde{\alpha}_y)(x_j)=\tilde{a}(x_j)=\sum_{i=1}^{r-1}a_{ij}t^i.
\]
Comparing coefficients yields that either $a_{ij}=0$, or $\mu^{\ell_j-i}=1$. As $i$ and $\ell_i$ lie in $\{1,\dots,r-1\}$, the latter is equivalent to $i=\ell_j$.
As $\tilde{a}$ preserves the $k[t]/(t^r)$-structure, i.\thinspace e.~$\tilde{a}\circ \tilde{\beta_y}=\tilde{r}$, we get that $\tilde{a}(x_0)=t$.
So using that $\tilde{a}$ is a $\kappa(y)$-morphism, i.\thinspace e.~that $\tilde{a}\circ \rho_2 =i_0\circ w^\#$, we get that
\begin{equation}\label{formula a}
\tilde{a}(p(x_0,x_1,\dots,x_{m}))=p(t,a_{1}t^{\ell_1},\dots,a_{m}t^{\ell_{m}}) 
\end{equation}
 for all $p(x_0,x_1,\dots, x_m)\in \kappa(y)[ x_0,x_1,\dots x_m]/\mathfrak{I}$, and for some $a_i\in A$.

Let $\tilde{a}: \kappa(y)[ x_0,x_1,\dots ,x_m]\to A[t]/(t^r)$ be defined by formula (\ref{formula a}).
For any generator $x_0^{s_0}x_1^{s_1}\dots x_m^{s_m}$ of $\mathfrak{I}$
\[
 \tilde{a}(x_0^{s_0}x_1^{s_1}\dots x_m^{s_m})=
a_1^{s_1}\dots a_m^{s_m}\
 t^{s_0+\ell_1s_1+\dots+\ell_ms_m}=t^{rs}=0\in A[t]/(t^r).
\]
This implies that $\mathfrak{I}\subset ker(\tilde{a})$.
Therefore, we get a unique well-defined map 
\[
{\tilde{a}: \kappa(y)[x_0,x_1,\dots,x_m]/ \mathfrak{I}\to A[t]/(t^r)}.
\]
Note that $\tilde{a}$ is a $\kappa(y)$-morphism, and preserves the $k[t]/(t^r)$-structure, and
one can check that $\alpha^{-1} \circ \tilde{a} \circ \tilde{\alpha}_y=\tilde{a}$, and $r_W\circ \tilde{a}=w^\#\circ \tilde{r}_y$.
Altogether, $\tilde{a}\in b^{-1}(y)(W)$ if and only if it is given by formula (\ref{formula a}).

\medskip

Now we are ready to construct a $\kappa(y)$-isomorphism 
$\beta: b^{-1}(y)\to \mathbb{A}^m_{\kappa(y)}$.
It suffices to give bijective, functorial maps
\[
\beta(W):b^{-1}(y)(W)\to \mathbb{A}^m_{\kappa(y)}(W)=\Hom_{\kappa(y)}(\kappa(y)[y_1,\dots,y_m],A)
\]
 for all affine $W=\Spec(A)\in (\Sch/\kappa(y))$.
Let $\beta (W)$ be the map which sends $\tilde{a}\in b^{-1}(W)$ given by formula (\ref{formula a})
to $a'\in \mathbb{A}^m_{\kappa(y)}(W)$ with
\begin{align*}
a':\ \kappa(y)[y_1,\dots,y_m]&\to A;
p(y_1,\dots,y_m)\mapsto p(a_1,\dots,a_m).
\end{align*}

This map is bijective, because there is an obvious inverse map.
It is easy to check that it is functorial.
\end{proof}

\begin{rem}
If we do not assume that the Galois extension $L/K$ is tame, then we cannot show Lemma \ref{main lemma}.
This is because Lemma \ref{ap1} is wrong in this case, see Example \ref{ex tame},
and hence we cannot show Lemma~\ref{ap2}, which is the main ingredient of the proof of Lemma \ref{main lemma}.
It would be very interesting to know what happens in the non-tame case.
\end{rem}

\begin{rem}
 Note that if we do not assume that $k=\bar{k}$, but that $L$ over $K$ is totally ramified, and that $k$ contains all $r$-th primitive roots of unity, one can still show Lemma \ref{main lemma}.\\
 There should also be no problem to replace $\mathcal{O}_K$ by a Henselian ring.
\end{rem}

\subsection{Sections of the Quotient}

\begin{cor}\label{sections of the quotient}
Assumptions and notation as in Theorem \ref{main theorem}.
Then ${\mathcal{X}(\mathcal{O}_K)=\emptyset}$ if and only if $\mathcal{Y}^G= \emptyset$.
If $\mathcal{Y}\to T$ is a weak N\'eron model of $X_L$, then $X(K)=\emptyset$ if and only if $\mathcal{Y}^G= \emptyset$.
\end{cor}

\begin{proof}
By Theorem \ref{main theorem}, $\mathcal{X}(\mathcal{O}_K)\cong \mathcal{Z}(\mathcal{O}_K)$.
As $\mathcal{Z}\to S$ is smooth and $\mathcal{O}_K$ is Henselian, $\mathcal{Z}(\mathcal{O}_K)=\emptyset$ if and only if $\mathcal{Z}_k(k)=\emptyset$ by \cite[Chapter 2.3, Proposition 5]{MR1045822}.
As $k$ is algebraically closed this is equivalent to $\mathcal{Z}_k=\emptyset$.
By Lemma \ref{main lemma} there is a surjective morphism $b:\mathcal{Z}_k\to \mathcal{Y}^G$,
so $\mathcal{Z}_k=\emptyset$ if and only if $\mathcal{Y}^G=\emptyset$. 

If $\mathcal{Y}\to T$ is a weak N\'eron model of $X_L$, $\mathcal{Z}\to S$ is a weak N\'eron model of $X$ by Theorem \ref{main theorem},
i.e.~in particular $\mathcal{Z}(\mathcal{O}_K)\cong X(K)$, which implies the second claim.
\end{proof}

Now we show one direction of the equivalence in Corollary \ref{sections of the quotient} without using $\mathcal{Z}$, because
this alternative proof yields an explicit construction of a section of the model $\mathcal{X}\to S$ through the image of a fixed point.

\begin{prop} \label{construction}
Assumptions and notation as in Theorem \ref{main theorem}.
If $\mathcal{Y}^G\neq \emptyset$, then $\mathcal{X}(\mathcal{O}_K)\neq \emptyset$.
\end{prop}

\begin{proof}
As $\mathcal{Y}^G\neq \emptyset$, there is a closed fixed point $y\in \mathcal{Y}$.
Note that $G$ acts on $\Spec(\hat{ {\mathcal{O}}} _{\mathcal{Y},y})$ given by some $\alpha_y\in \Aut(\hat{ {\mathcal{O}}} _{\mathcal{Y},y})$ with $\alpha_y^r=\Id$,
such that the natural map $j: \Spec(\hat{ {\mathcal{O}}} _{\mathcal{Y},y}) \to \mathcal{Y}$ is $G$-equivariant.
As $L/K$ is a tame Galois extension, $G$ is a cyclic group of order prime to $\Char(k)$, so by Lemma \ref{ap1},
$G$ acts on $R:=\mathcal{O}_L$ given by some $\alpha_R \in \Aut(R)$ sending a generator $t$ of the maximal ideal in $R$ to $\mu t$, with $\mu \in R$ a primitive $r$-th root of unity.
Note that $\hat{\mathcal{O}}_{\mathcal{Y},y}$ is an $R$-module via $\beta_y:=(\varphi \circ j)^\#$,
and $\beta_y$ is $G$-equivariant.
As $\varphi$ is smooth, and the residue field of $R$ is equal to the residue field of $\hat{\mathcal{O}}_{\mathcal{Y},y}$,
\cite[Proposition~17.5.3]{ MR0238860} implies that $\hat{\mathcal{O}}_{\mathcal{Y},y}\cong R\Pol \tilde{x}_1,\dots,\tilde{x}_n\Por$ as $R$-module for some $ \tilde{x}_1,\dots,\tilde{x}_n \in\hat{\mathcal{O}}_{\mathcal{Y},y}$.
Note that $t,\tilde{x}_1, \dots, \tilde{x}_n$ form a regular system of parameters of $\hat{\mathcal{O}}_{\mathcal{Y},y}$.
As $\alpha_y(t)=\alpha_R(t)=\mu t$, by Lemma \ref{ap1} we may choose a system of parameters $x_0,\dots,x_n$ with $\alpha_y(x_i)=\mu^{\ell_i}x_i$ for some $\ell_i\in \mathbb{N}$, such that $x_0=t$.
So $\hat{\mathcal{O}}_{\mathcal{Y},y}\cong R\Pol\tilde{x}_1,\dots ,\tilde{x}_n \Por \cong R\Pol x_1,\dots, x_n\Por$ as $R$-modules.
Let $I\subset \hat{\mathcal{O}}_{\mathcal{Y},y}$ be the ideal generated by $x_1,\dots, x_n$.
Note that $\alpha_y ( I) \subset I$. So the quotient map 
\[
 \hat{\sigma}: \hat{\mathcal{O}}_{\mathcal{Y},y}\to \hat{\mathcal{O}}_{\mathcal{Y},y}/I=R\Pol x_1,\dots, x_n \Por/(x_1,\dots, x_n)\cong R
\]
is a $G$-equivariant retraction of $\beta_y$.
Therefore $\hat{\sigma}^\#$ is a section of $\varphi \circ j$, and $\sigma:=j\circ \hat{\sigma}^\#$ is a section of $\varphi$.
As both $\hat{\sigma}$ and $j$ are $G$-equivariant, the same holds for $\sigma$.

Let the $G$-action on $T$ be given by $g_T\in \Aut(T)$, that on $\mathcal{Y}$ by $g\in \Aut(\mathcal{Y})$.
Let $\pi: \mathcal{Y}\to \mathcal{X}$ and $\pi_T : T\to S$ be the quotients.
Let $\varphi_G : \mathcal{X}\to S$ be the structure map of $\mathcal{X}$ as $S$-scheme.
Every element in $\mathcal{X}(\mathcal{O}_K)$ is given by a section of $\varphi_G$.
As $\sigma$ is $G$-invariant and $\pi$ is a quotient map, $\pi \circ \sigma \circ g_T =\pi \circ g \circ \sigma =\pi \circ \sigma$.
So by the universal property of the quotient $\pi_T:T\to S$, there exists a unique $\sigma_G: S \to \mathcal{X}$ such that $\pi \circ \sigma = \sigma_G \circ \pi_T$.
Furthermore, 
\[
\varphi_G \circ \sigma_G \circ \pi_T= \varphi_G \circ \pi \circ \sigma = \pi_T \circ \varphi \circ \sigma= \pi_T \circ id_T = \pi_T.
\]
As $\pi_T$ is an epimorphism, $\varphi_G \circ \sigma_G = id_S$, i.e. $\sigma_G$ is a section of $\varphi_G$.
\end{proof}

Note that  the image of a closed fixed point $y\in \Sm(\mathcal{Y}/T)^G$ in $\mathcal{X}$ is a singular point in general,
so in fact we construct sections through singular points.
Here is an example for such a section through a singular point.

\begin{ex}
\label{exsec}
Assumptions and notation as in Example \ref{the example}.
The closed point $Q=(0,0)\in\mathcal{Y} = \mathbb{A}^1_{k\Pol t \Por}$ is fixed, and
the $k\Pol s \Por$-scheme 
\[
\mathcal{Y}/G \cong \Spec(k\Pol s\Por [b,c]/(sc-b^2))
\]
is singular in the image $Q'=(0,0,0)$ of $Q$ under the quotient map.
The proof of Proposition \ref{construction} implies that there is a section $\sigma_G$ of $\mathcal{Y}/G \to \Spec (k\Pol s \Por)$ through $Q'$.
Such a section is for example given by 
\[
 {\sigma_G}^\#(P(s,b,c))=P(s,0,0)\in k\Pol s\Por.
\]
Note that the $G$-equivariant section $\sigma$ of $\mathcal{Y}\to \Spec(k \Pol t \Por)$ which descends to $\sigma_G$ is given by
$ \sigma^\#(P(t,x))=P(t,0)$.
\end{ex}

\section{Motivic Invariants}
\subsection{Motivic Serre Invariant}

\begin{defn}
The \emph{Grothendieck group of $k$-varieties} $K_0(\Var_k)$ is defined to be the abelian group with
\begin{itemize}
 \item
generators: isomorphism classes $[U]$ of separated $k$-schemes $U$ of finite type
  \item 
relations: $[U]=[U\setminus V]+[V]$ for every closed immersion $V\hookrightarrow U$ (\emph{scissor relations})
\end{itemize}
The product $[U][V]=[U\times_{\Spec (k)}V]$ defines a ring structure on $K_0(\Var_k)$. We call this ring the \emph{Grothendieck ring of $k$-varieties}.\\
Set $\mathbb{L}\!:=[\mathbb{A}^1_k]$.

The \emph{modified Grothendieck ring of $k$-varieties} $K_0^{\Bla}(\Var_k)$
is the quotient of $K_0(\Var_k)$ by the ideal generated by elements
\[
 [U]-[V]
\]
where $U$ and $V$ are separated $k$-schemes of finite type such that there exists a finite, surjective, purely inseparable $k$-morphism $U\to V$.\\
We still denote the image of $\mathbb{A}^1_k$ in $K_0^{\Bla}(\Var_k)$ by $\mathbb{L}$.
\[
K_0^{\mathcal{O}_K}(\Var_k) \!:=
\begin{cases}
K_0(\Var_k) &  \text{if }\mathcal{O}_K\text{ has equal characteristic}\\
 K_0^{\Bla}(\Var_k) &  \text{if }\mathcal{O}_K\text{ has mixed characteristic}
\end{cases}
\]
\end{defn}

\begin{defn}
\label{def serre invariant}
Let $X$ be a smooth $K$-variety with weak N\'eron model $\mathcal{X}\to S$. Then the \emph{motivic Serre invariant} $S(X)$ is defined by
\[
 S(X)\!:=[\mathcal{X}_k]\in K_0^{\mathcal{O}_K}(\Var_k)/(\mathbb{L}-1).
\]
% with $\mathcal{X}_k$ the special fiber of $\mathcal{X}\to S$.
By \cite[Proposition-Definition 3.6]{000} this definition does not depend on the choice of a weak N\'eron model.
\end{defn}

\begin{rem}
\label{S=0}
Let $X$ be a smooth, separated $K$-variety without $K$-rational point.
Then $S(X)=0$.
This holds, because in this case $X$ viewed as an $S$-scheme is a weak N\'eron model of $X$, i.\thinspace e.~the special fiber of this weak N\'eron model is empty.
Hence if $S(X)\neq 0$, then $X$ has a $K$-rational point.
\end{rem}

\begin{thm}\label{tolle idee}
Let $X$ be a smooth, proper $K$-variety.
Let $L/K$ be a tame Galois extension, $\mathcal{O}_L$ the ring of integers of $L$, $T:=\Spec(\mathcal{O}_L)$.
Let $\varphi:\mathcal{Y}\to T$ be a weak N\'eron model of ${X}_L$ with a good
$G\!:=\Gal(L/K)$-action  extending the Galois action on $X_L$. Then
\[
S(X)=[\mathcal{Y}^G]\in K_0^{\mathcal{O}_K}(\Var_k)/(\mathbb{L}-1).
\]
\end{thm}

\begin{proof}
By Theorem \ref{main theorem} we know that $\mathcal{Z}\to S$ as constructed in Construction~\ref{construction Z} is a weak N\'eron model of $X$.
Hence by definition
 $S(X)$ equals to the class of the special fiber $\mathcal{Z}_k$ in $K_0^{\mathcal{O}_K}(\Var_k)/(\mathbb{L}-1)$,
and it suffices to show the following statement:
\[
[\mathcal{Z}_k]=[\mathcal{Y}^G]\in K_0^{\mathcal{O}_K}(\Var_k)/(\mathbb{L}-1).
\]
As $K_0^{\mathcal{O}_K}(\Var_k)/(\mathbb{L}-1)$ is a quotient of  $K_0(\Var_k)/(\mathbb{L}-1)$, it suffices to show
the equation in
$K_0(\Var_k)/(\mathbb{L}-1)$.

Consider $b:\mathcal{Z}_k\to \mathcal{Y}^G$ as in Lemma \ref{main lemma}.
We can find $U_i \subset \mathcal{Y}^G$ such that
$ \mathcal{Y}^G=U_1\sqcup \dots \sqcup U_m$ with
${U_i\subset \mathcal{Y}^G\setminus (\cap_{j<i}U_j)}$ open, and $b^{-1}(U_i)\cong \mathbb{A}^{m_i}_{U_i}$ for some $m_i\in \mathbb{N}$, by proceeding in the following way.

By \cite[Proposition 3.5]{ MR1149171}, $\mathcal{Y}^ G$ is smooth over $T^G=\Spec(k)$, hence in particular reduced.
Replacing $\mathcal{Y}^ G$ by an open subset, we may assume that it is integral.
Let $\eta$ be the generic point of $\mathcal{Y}^G$ with residue field $\kappa(\eta)$.
By Lemma \ref{main lemma}
there is an isomorphism $\beta: b^{-1}(\eta)\to \mathbb{A}^{m_1}_{\kappa(\eta)}$
over $\kappa(\eta) $ for some $m_1\in \mathbb{N}$.
As $\beta$ is defined by finitely many rational functions over $\mathcal{Y}^ G$,
we find an open subset $U_1$ of $\mathcal{Y}^ G$ over which $\beta$ is already defined.
In particular $b^ {-1}(U_1)\cong \mathbb{A}^{m_1}_{U_1}$.
Now one can proceed with $\mathcal{Y}^ G\setminus U_1$ in the same way.
The claim follows by noetherian induction using that $\mathcal{Y}^G$ is of finite type over $k$.

So using the scissor relations in the Grothendieck ring of $k$-varieties we get in $K_0(\Var_k)/(\mathbb{L}-1)$ that
\begin{align*}
[\mathcal{Z}_k] =[b^{-1}(\mathcal{Y}^G)]&=[b^{-1}(U_1)\sqcup \dots \sqcup b^{-1}(U_m)]\\
&=[b^{-1}(U_1)]+[b^{-1}(U_2)\sqcup \dots \sqcup b^{-1}(U_m)]
=\dots\\
&=\sum_{i=1}^m [b^{-1}(U_i)]
= \sum_{i=1}^m[\mathbb{A}_{U_i}^{m_i}]
= \sum_{i=1}^m [\mathbb{A}_k^{m_i}][U_i]
= \sum_{i=1}^m[U_i]\\
&= [U_1\sqcup \dots \sqcup U_m]
=[\mathcal{Y}^G].
\end{align*}
This proves Theorem \ref{tolle idee}.
\end{proof}

\subsection{Rational Volume}

\begin{fact}\cite[Example 4.3 and Corollary 4.14]{001}\\
There exists a unique ring morphism (realization morphism)
\[
 \chi_c: K_0^{\mathcal{O}_K}(\Var_k)/(\mathbb{L}-1)\to \mathbb{Z}
\]
that sends a class of a separated $k$-scheme $U$ of finite type to the Euler characteristic with proper support
\[
 \chi_c(U)=\sum_{i\geq 0}(-1)^i\Dim H^i_c(U,\mathbb{Q}_l),
\]
with $l\neq \Char(k)$ a prime. The map does not depend on the choice of $l$.
\end{fact}

\begin{defn}
\label{def rational volume}
Let $X$ be a smooth $K$-variety with weak N\'eron model. Then the \emph{rational volume} of $X$ is defined by
\[
 s(X)\!:=\chi_c(S(X))\in \mathbb{Z}.
\]
\end{defn}

\begin{rem}
\label{rem rat vol}
Let $X$ be a smooth $K$-variety without $K$-rational point.
Then $s(X)=0$.
This holds, because by Remark \ref{S=0},
$S(X)=0$, hence in particular ${s(X)=\chi_c(S(X))=0}$.
So if $s(X)\neq 0$, then $X$ has a $K$-rational point.
\end{rem}

\begin{thm} \label{rational volume}
Let $X$ be a smooth, proper $K$-variety,
and let $L/K$ be a tame Galois extension, such that
$G\! :=\Gal(L/K)$ is a $q$-group, $q \neq \Char(k)$ a prime.
Then
\[
 s(X_L)=s(X) \mod q.
\]
In particular, if  $s(X_L)$ does not vanish modulo $q$, then $X$ has a $K$-rational point.
\end{thm}

\begin{proof}
Let $\mathcal{O}_L$ be the ring of integers of $L$, $T:=\Spec(\mathcal{O}_L)$.
By Theorem \ref{gaction on wnm} there is a weak N\'eron model $\varphi: \mathcal{\mathcal{Y}}\to T$ of $X_L$ with a good $G$-action on $\mathcal{\mathcal{Y}}$,
extending the Galois action on $X_L$. Hence
Theorem \ref{tolle idee} implies that
\[
 S(X)=[\mathcal{\mathcal{Y}}^G] \in K_0^{\mathcal{O}_K}(\Var_k)/(\mathbb{L}-1).
\]
As $X_L\subset \mathcal{Y}$ is $G$-invariant, the same holds for $\mathcal{Y}_k$, so the action of $G$ on $\mathcal{Y}$ restricts to $\mathcal{Y}_k$.
By
\cite[Proposition 5.4]{1009.1281},
for every variety $U$ over a field $F$ with a good $G$-action, $ \chi_c(U)= \chi_c(U^G) \mod q$.
This Proposition is based on an argument in \cite[Section 7.2]{MR2555994}.
In our case we get
\[
 \chi_c(\mathcal{\mathcal{Y}}_k)= \chi_c(\mathcal{\mathcal{Y}}_k^G) \mod q.
\]
As $\mathcal{\mathcal{Y}}^G\subset \mathcal{\mathcal{Y}}_k$, see the proof of Lemma \ref{main lemma}, $\mathcal{\mathcal{Y}}^G=\mathcal{\mathcal{Y}}_k^G$.
As $\mathcal{\mathcal{Y}}$ is a weak N\'eron model of $X_L$, by definition $S(X_L)=[\mathcal{\mathcal{Y}}_k]\in K_0^{\mathcal{O}_K}(\Var_k)$.
So altogether we obtain
\begin{align*}
 s(X_L)=\chi_c(S(X_L))=\chi_c(\mathcal{\mathcal{Y}}_k)&=\chi_c(\mathcal{\mathcal{Y}}^G) \mod q\\
&= \chi_c(S(X)) \mod  q\\
 &=s(X) \mod q.
\end{align*}
Assume now that $s(X_L)\neq 0\mod q$. This implies that $s(X)\neq 0$.
But the rational volume of a smooth $K$-variety without $K$-rational point vanishes, see Remark \ref{rem rat vol}, hence $X$ has a $K$-rational point.
\end{proof}

\section{Rational Points on Certain Varieties with Potential Good Reduction}
\begin{defn}
 A smooth, proper $K$-variety $X$ has \emph{potential good reduction} (\emph{after a base change of order $r$})
if there exists a Galois extension $L/K$ (of degree $r$),
such that
$X_L$
has a smooth and proper model.
\end{defn}

\begin{cor}\label{korollar rational volume}
 Let $X$ be a smooth, proper $K$-variety,
which has potential good reduction after a base change of order $q^r$, $q\neq char(k)$ a prime.
Then
\[
\chi(X):=\sum_{i\geq 0}(-1)^i\Dim H^i(X_{K^s},\mathbb{Q}_l)=s(X) \mod q
\]
with $K^s$ a separable closure of $K$, $l\neq \Char(k)$ a prime.
In particular, if $\chi(X)$ does not vanish modulo $q$, then $X$ has a $K$-rational point.
\end{cor}

\begin{proof}
Let $L/K$ be the field extension of degree $q^r$, such that there is a smooth and proper model of $X_L$.
Let $\mathcal{O}_L$ be the ring of integers of $L$, $T:=\Spec(\mathcal{O}_L)$, and $\varphi: \mathcal{Y}\to T$ a smooth and proper model of $X_L$,
which is in particular a weak N\'eron model of $X_L$.
So by definition $s(X_L)=\chi_c(\mathcal{Y}_k)$.
As $\varphi$ is proper, $\mathcal{Y}_k$ is proper over $k$, and hence the ordinary cohomology coincides with the cohomology with proper support, i.\thinspace e.~$\chi_c(\mathcal{Y}_k)=\chi(\mathcal{Y}_k)$.
As $\varphi$ is proper and smooth, by \cite[Expos\'e V, Theorem 3.1]{MR0463174} we get bijections between
$H^i(\mathcal{Y}_k,\mathbb{Z}/n\mathbb{Z})$, and $H^i(\mathcal{Y},\mathbb{Z}/n\mathbb{Z})$, and $H^i(X_L\times_{\Spec(L)} \Spec(L^s),\mathbb{Z}/n\mathbb{Z})$ for all $i$,
with $L^s$ a separable closure of $L$.
Therefore we have for all $i$ that
\[
 \Dim H^i(\mathcal{Y}_k,\mathbb{Q}_l)=\Dim H^i(X_L\times_{\Spec(L)}L^s,\mathbb{Q}_l).
\]
Note that $L^s=K^s$, because $L/K$ is a tame Galois extension. Therefore we get
${X_L\times_{\Spec(L)}\Spec(L^s)=X\times_{\Spec(K)}\Spec(K^s)=X_{K^s}}$,
hence
\begin{align*}
 \chi(X)&=\sum_{i\geq 0}(-1)^i\Dim H^i(X_{K^s},\mathbb{Q}_l)
=\sum_{i\geq 0}(-1)^i\Dim H^i(\mathcal{Y}_k,\mathbb{Q}_l)=\chi(\mathcal{Y}_k).
\end{align*}
This implies that $s(X_L)=\chi(X)$. 
Hence Theorem \ref{rational volume} implies the corollary.
\end{proof}

\begin{cor}\label{Korollar}
Let $X$ be a smooth, proper $K$-variety, let $L/K$ be a tame Galois extension of prime order q,
and assume that there is a smooth and proper model of $X_L$ with a good $G:=\Gal(L/K)$-action extending the Galois action on $X_L$,
i.e.~in particular $X$ has potential good reduction after a base change of order $q$.

 If $\chi(X,\mathcal{O}_X)$ does not vanish modulo $q$, then $X$ has a $K$-rational point.
\end{cor}

\begin{proof}
Let $\mathcal{O}_L$ be the ring of integers of $L$, $T:=\Spec(\mathcal{O}_L)$.
Let $\varphi : \mathcal{Y}\to T$ be the smooth and proper model of $X_L$ on which there is a good $G$-action extending the Galois action on $X_L$.
As $X_L\subset \mathcal{Y}$ is $G$-invariant, the same holds for $\mathcal{Y}_k\subset \mathcal{Y}$, hence
the $G$-action on $\mathcal{Y}$ restricts to a good $G$-action on $\mathcal{Y}_k$.
Let $f: \mathcal{Y}_k\to  \mathcal{Y}_k/G$ be the quotient.\\
Assume that the action of $G$ on $\mathcal{Y}$ has no fixed point, so in particular the action of $G$ on $\mathcal{Y}_k$ has no fixed point.
As $q$ is a prime, the action is free.
So $f$ is a finite, \'etale morphism of degree $q$ by \cite[Expos\'e V, Corollaire 2.3]{MR0217087}.
Moreover $\mathcal{Y}_k$ is smooth and proper over $k$, because $\varphi$ is smooth and proper.
As $f$ is \'etale and finite, $\mathcal{Y}_k/G$ is smooth and proper over $k$, too.
As $f$ is \'etale, $f^*(T_{\mathcal{Y}_k/G})=T_{\mathcal{Y}_k}$, and therefore $f^*(\Td(T_{\mathcal{Y}_k/G}))=\Td(T_{\mathcal{Y}_k})$.
Let ${s: \mathcal{Y}_k\to \Spec(k)}$ and $s':{\mathcal{Y}_k/G}\to \Spec(k)$ be the structure maps. We have ${s=s'\circ f}$.
Using \cite[Corollary 15.2.2]{MR1644323}, and the projection formula in the third line, we obtain
\begin{align*}
\chi ({\mathcal{Y}}_k,\mathcal{O}_{\mathcal{Y}_k})&= s_*(\Ch(\mathcal{O}_{\mathcal{Y}_k})\Td(T_{\mathcal{Y}_k}))\\
&= s'_*(f_*(\Ch(\mathcal{O}_{\mathcal{Y}_k})f^*(\Td(T_{{\mathcal{Y}_k/G}}))))\\
&= s'_*(f_*(\Ch(\mathcal{O}_{\mathcal{Y}_k}))\Td(T_{{\mathcal{Y}_k/G}}))\\
&= s'_*(\Deg(f)\Ch(\mathcal{O}_{{\mathcal{Y}_k/G}})\Td(T_{{\mathcal{Y}_k/G}}))= q\ \chi ({\mathcal{Y}_k/G},\mathcal{O}_{{\mathcal{Y}_k/G}}).
\end{align*}
This implies that $\chi (\mathcal{Y}_k,\mathcal{O}_{\mathcal{Y}_k})=0 \mod q$.

Note that
${H^i(X_L,\mathcal{O}_{X_L})=H^i(X_L,\mathcal{O}_{X}\otimes_K L)=H^i(X,\mathcal{O}_{X})\otimes _K L}$ holds for all ${i\geq 0}$.
% The second equation follows from \cite[Chapter 5, Corollary 5]{MR2514037}.
So in particular $\chi(X_L,\mathcal{O}_{X_L})=\chi(X,\mathcal{O}_X)$.
As $\varphi$ is smooth and proper,
and $T$ is connected, by \cite[Theorem~7.9.4.I]{MR0217084} the Euler characteristic is constant on the fibers of $\varphi$, hence 
$\chi(\mathcal{Y}_k,\mathcal{O}_{\mathcal{Y}_k})=\chi(X_L,\mathcal{O}_{X_L})$, which does not vanish modulo $q$.
This is a contradiction, hence $\mathcal{Y}^G\neq \emptyset$.
So Corollary \ref{sections of the quotient} implies that $X$ has a $K$-rational point.
\end{proof}

\section{Appendix}

In this section we show two lemmas concerning tame cyclic actions on Henselian, regular, local rings.
These results are used in Section \ref{local studies}.

The following lemma should be known to the experts;
a similar statement can be found in \cite{MR0234953}.

\begin{lem} \label{ap1} 
Let $A$ be a regular, Henselian ring of dimension n with maximal ideal $\Em$,
such that its residue field $\kappa$ is a field of $\Char(\kappa)\nmid r$ containing all $r$-th roots of unity, and
let $\alpha \in \Aut(A)$ with $\alpha ^r=\Id$, such that the residual map on $\kappa$ is trivial.\\
There exists a regular system of parameters $x_1,\dots, x_n \in\Em \subset A $ with
${{\alpha} (x_i)=\mu^ {\ell_i} x_i}$, 
$\mu\in A$ a primitive r-th root of unity, and $\ell_i \in \{0,\dots,r-1\}$.\\
If there are $z_1,\dots, z_s\in \Em\subset A$, such that the $\bar{z}_1,\dots,\bar{z}_s\in \Em\!/\!\Em^2$ are linearly independent,
and such that ${\alpha}(z_i)=\mu^{\ell _i} z_i$ for some $\ell_i\in {\{0,\dots,r-1\}}$,
then we may chose $x_i=z_i$ for $i\leq s$.
\end{lem}

\begin{proof}
Consider the polynomial $p(x):=x^r-1 \in A[x]$.
Let $\mu \in \kappa$ be an $r$-th root of unity, hence
$p(\mu)=0\in \kappa$. $p'(\mu)=r\mu^{r-1}\neq 0\in \kappa$, because $r\neq 0\in \kappa$.
As $A$ is Henselian, Hensel's Lemma gives us a $\tilde{\mu}\in A$, such that $\tilde{\mu}=\mu \mod \Em$, and $p(\tilde{\mu})=0$,
i.\thinspace e.~$\tilde{\mu}$ is a lift of $\mu$, and $\tilde{\mu}^r=1$.
So we may fix a primitive r-th root of unity $\mu\in A$.
Identify $\mu$ with its image in $\kappa$ under the residue map.

As $A$ is a regular local ring of dimension $n$ with residue field $\kappa$, $\Em \! / \!\Em ^2$ is an $n$-dimensional $\kappa$-vector space.
Let $\bar{\alpha}\in \Aut(\Em \! / \!\Em ^2 )$ be the automorphism induced by $\alpha$.
As the morphism on $A/\!\Em=\kappa$ induced by $\alpha$ is trivial, $\bar{\alpha}$ is a $\kappa$-linear map.
For some algebraic closure $\bar{\kappa}$ of $\kappa$, there exists a basis  of $\Em \! / \!\Em ^2\otimes_\kappa \bar{\kappa}$, such that the matrix corresponding to $\bar{\alpha}$ has Jordan normal form.
As $\bar{\alpha}^r=\Id$, all eigenvalues are $r$-th roots of unity, i.\thinspace e.~powers of $\mu$, and as $r\neq 0\in \kappa\subset \bar{\kappa}$, the matrix is already diagonal.
But all $r$-th roots of unity are assumed to be in $\kappa$,
so $\bar{\alpha}$ is diagonalizable, too.
Therefore $\Em\! / \! \Em^2$ decomposes into eigenspaces $E_j$.
By assumption, for all $i$ there exists a $j$ such that $\bar{z}_i\in E_j$. Note that for all $j$ one can choose a basis $B_j$ of $E_j$ such that for all $i$, $\bar{z}_i\in \cup B_j$.
This uses the fact that the $\bar{z_i}$ are linearly independent.
Set  $\{\bar{x}_{s+1},\dots \bar{x}_{n}\}\!:=\cup B_j \!\setminus\! \{\bar{z}_1,\dots,\bar{z}_s\}$.
As the $E_j$ are eigenspaces, we have $\bar{\alpha}(\bar{x}_{i})=\mu^{l_i}\bar{x}_{i}$ for some $l_i\in \{0,\dots,l-1\}$.

Choose $\tilde{x}_i\in A$ such that $\tilde{x}_i\!\mod \Em^2=\bar{x}_i$.
As $r$ is invertible in $A$, we can define $x_i$ for $i>s$ as follows.
\[
x_i:=\frac{1}{r}\sum_{j=0}^{r-1}\mu^{-\ell_ij}\alpha^j(\tilde{x}_i) 
\]
We have
\begin{align*}
 \alpha(x_i)&=\frac{1}{r}\sum_{j=0}^{r-1}\mu^{-\ell_ij}\alpha^{j+1}(\tilde{x}_i)
 =\frac{\mu^{\ell_i}}{r}\sum_{j=0}^{r-1}\mu^{-\ell_i(j+1)}\alpha^{j+1}(\tilde{x}_i)\\
  &=\frac{\mu^{\ell_i}}{r}(\sum_{j=1}^{r-1}\mu^{-\ell_ij}\alpha^j(\tilde{x}_i)+\mu^{-\ell_ir}\alpha^r(\tilde{x}_i))=\mu^{\ell_i} x_i.
\end{align*}
Moreover
\[
 x_i\!\mod \Em^2 = \frac{1}{r}\sum_{j=0}^{r-1}\mu^{-\ell_ij}\mu^{\ell_ij}(\bar{x}_i)=\frac{1}{r}\sum_{j=0}^{r-1}\bar{x}_i =\bar{x}_i,
\]
hence $\{z_1,\dots,z_s,x_{s+1},\dots,x_n\}$ is a system of regular parameters in $A$.

\end{proof}

\begin{rem}\label{rem not henselian}
 In fact we do not need to assume in Lemma \ref{ap1} that $A$ is Henselian, but only that all the $r$-th roots of unity in $\kappa$ lift to $A$. 
The same is true in Remark~\ref{rem ap2} and Lemma~\ref{ap2}.
 \end{rem}

If we do not assume that $r$ is prime to $\Char(\kappa)$, Lemma \ref{ap1} is wrong. To see this, look at the following example:

\begin{ex}
\label{ex tame}
Let $\kappa$ be an algebraically closed field with $\Char(\kappa)=2$.
Then ${A\!:=\kappa\Pol x,y\Por}$ is a complete local ring with maximal ideal $\Em=(x,y)\subset A$.
Let $\alpha\in \Aut(A)$ with
$\alpha(P(x,y))=P(x,x+y)$ for all $P(x,y)\in A$.
We have that $\alpha^2(P(x,y))=P(x,2x+y)=P(x,y)$, because $\Char (\kappa)=2$, 
hence $\alpha ^2=\Id$.
Note that $\bar{\alpha}:\Em/\Em^2\to\Em/\Em^2 $ is not diagonalizable.
\end{ex}

Let $A$ be a ring, $\alpha\in \Aut(A)$ with $\alpha^r=id$. Then $\alpha$ defines an action of $G:=\mathbb{Z}/r\mathbb{Z}$ on $A$, and the subring
\[
 A^G\!:=\{a\in A\mid \alpha(a)=a\}\subset A
\]
is called the \emph{ring of invariants}.

\begin{rem}\label{rem ap2}
Let $A$ be as in Lemma \ref{ap1}. Then $\Em^G\!:=\Em\cap A^G\subset A^G$
is an ideal and we have
$A^G/\Em^G\hookrightarrow A/\Em=\kappa$.
With a proof similar to the proof of Lemma \ref{ap1},
we can show that there exists a lift $\tilde{s}\in A$ of  $s\in \kappa$ such that $\alpha(\tilde{s})=\tilde{s}$,
i.\thinspace e.~$\tilde{s}\in A^G$, and hence $A^G/\Em^G=\kappa$.
Hence there is a ringhomomorphism $A^G\to \kappa$, and hence we may consider $\kappa \otimes_{A^G}A$.

Note that $G:=\mathbb{Z}/r\mathbb{Z}$ acts on $\kappa\otimes_{A^G}\!A$ given by $\Id\otimes\alpha\in \Aut(\kappa\otimes_{A^G}\!A)$, such that the canonical maps $\rho_1:A\to \kappa\otimes_{A^G}A$ and $\rho_2: \kappa\to \kappa\otimes_{A^G}A$ are $G$-equivariant for this $G$-action, and the given $G$-action on $A$, and the trivial $G$-action on $\kappa$, respectively.
\end{rem}

\begin{lem}\label{ap2}
Let $A$ be as in Lemma \ref{ap1}. 
Then $ \kappa\otimes_{A^G}\!A\cong \kappa[x_1,\dots,x_m]/\mathfrak{I}$, $m\leq n$, and
\[
(\Id\otimes\alpha)(p(x_1,\dots,x_m))=p(\mu^{\ell_0}x_1,\dots,\mu^{\ell_m}x_m)
\]
for some $\ell_i\in \{1,\dots, r-1\}$, $p(x_1,\dots,x_m)\in \kappa\otimes_{A^G}\!A$, $\mu \in \kappa$ a primitive $r$-th root of unity, and
$\mathfrak{I}\subset \kappa[x_1,\dots, x_m]$ is the ideal generated by monomials of the form $x_1^{s_1}\dots x_m^{s_m}$ with $s_1\ell_1+\dots+s_m\ell_m=sr$, $s\in \mathbb{N}$.
\end{lem}

\begin{proof}
Set $\tilde{A}\!:=\kappa\otimes_{A^G}\!A$.
Consider
$ \rho_1: A\to \tilde{A};\ a\mapsto 1\otimes a$.
Take any $s\in \kappa$ and $a\in A$. As above we can chose a lift $\tilde{s}\in A^G$ of $s$.
Hence we get $\rho_1(\tilde{s}a)=1\otimes \tilde{s}a=s\otimes a$, hence $\rho_1$ is surjective.
Note that $0=\rho_1(a)=1\otimes a$ for some $a\in A$ if and only if we can write $a=a_1a_2$ for some $a_1\in A^G$, $a_2\in A$, and $r^G(a_1)=0$, i.\thinspace e.~$a_1\in \Em^G\!\!:=\Em\cap A^G$,
hence $ \Ker(\rho_1)=A\Em^G$.\\
By Lemma \ref{ap1} there exists a system of parameters $y_1,\dots,y_n\in A$ such that ${\alpha(y_i)=\tilde{\mu}^{\ell_i}y_i}$, $\ell_i\in \{0,\dots, r\}$, $\tilde{\mu} \in A$ a primitive $r$-th root of unity, which is a lift of $\mu\in \kappa$.
So $A\Em^G\subset A$ is the ideal generated by monomials of the form $y_1^{s_1}\dots y_n^{s_n}$ with $s_1\ell_1+\dots+s_n\ell_n=sr$, $s\in \mathbb{N}$.
As $\Em^{nr}$ is generated by monomials of degree $nr$ in the $y_i$, all generators are divisible by $y_i^r$ for at least one $i$.
Note that for all $i$, $y_i^r\in \Em^G$. Hence $\Em^{nr}\subset A\Em^G$.\\
Set $N\!:=nr$. So $\tilde{A}\cong \kappa\otimes_{{A}^G}\!(A/\!\Em^N)$.
We show by induction that this is generated as a $\kappa$-algebra by the images of the $y_i$.
The induction assumption is clear, because in this case $\kappa\otimes_{{A}^G}\!(A/\!\Em^1)\cong \kappa$.
Assume that $\kappa\otimes_{{A}^G}\!(A/\!\Em^{l})$ is generated as a $\kappa$-algebra by the images of the $y_i$.
Let $\tilde{A}_{l+1}$ be the subalgebra of ${\kappa\otimes_{{A}^G}\!(A/\!\Em^{l+1})}$ generated by the images of the $y_i$.
Take any element $1\otimes a$ in $\kappa\otimes_{A^G}\!(A/\!\Em^{l+1})$.
By the induction assumption there is an $\tilde{a}\in A/\!\Em^{l+1}$, such that $a-\tilde{a}\in \Em^l\!/\!\Em^{l+1}$,
and $1\otimes \tilde{a}\in \tilde{A}_{l+1}$.
Note that $\Em^l/\Em^{l+1}$ is a $\kappa$-vector space generated by monomials of degree $l$ in the $y_i$.
So $1\otimes (\tilde{a}-a)\in \tilde{A}_{l+1}$, and therefore the same holds for $1\otimes a= 1\otimes \tilde{a}+1 \otimes (a-\tilde{a})$.
Altogether, $\tilde{A}$ is generated as a $\kappa$-algebra by the images of the $y_i$.\\
Let $x_1,\dots, x_m$ be the images of those $y_i$ with $\ell_i\neq 0$.
Note that, if $\ell_i=0$, $y_i\in \Em^G\subset \Ker(\rho_1)$, i.\thinspace e.~$\rho_1(y_i)=0$. Hence the $x_i$ generate $\tilde{A}$ as a $\kappa$-algebra.
Renumbering the $y_i$, we may assume that $\rho_1(y_i)=x_i$.
We have
\[
 (\Id\otimes\alpha)(x_i)=(\Id\otimes\alpha)(1\otimes y_i)=1\otimes\alpha(y_i)=1\otimes \tilde{\mu}^{\ell_j}y_i=\mu^{\ell_i}x_i.
\]
 Moreover, using $\Ker(\rho_1)=A\Em^G$, we obtain that
\[
 \tilde{A}\cong \kappa[x_1,\dots,x_m]/{\mathfrak{I}}
\]
with $\mathfrak{I}$ generated by $x_1^{s_1}\dots x_m^{s_m}$ with $s_1\ell_1+\dots+s_m\ell_m=sr$, $s\in \mathbb{N}$.
As $(\Id\otimes\alpha)$ is a $\kappa$-morphism, $(\Id\otimes\alpha)(p(x_1,\dots,x_m))=p(\mu^{\ell_1}x_1,\dots,\mu^{\ell_m}x_m)$ with $\ell_i\in \{1,\dots, r\}$ for $p(x_1,\dots, x_m)\in \tilde{A}$.
\end{proof}

\begin{rem}
Note that if $A$ is of mixed characteristic, it is not a $\kappa$-algebra, but
$A\otimes_{A^G}\kappa$ is. As we tensor over $A^G$, we keep the information concerning the $G$-action on $A$.
\end{rem}

\bibliographystyle{babalpha}

	\bibliography{pubref}

\end{document}